\documentclass[11pt]{article}
 \usepackage{amsmath, amsthm, amssymb, bbm, setspace,bigints}
 \usepackage[margin=1 in]{geometry}
\usepackage{caption}
\usepackage{subcaption}
\usepackage[toc,page]{appendix}
\usepackage{cite}         
\usepackage[pdftex]{graphicx}
\usepackage{pdfpages}
\usepackage{epstopdf}
\usepackage{booktabs}
\usepackage{authblk}

\newcommand{\be} {\begin{eqnarray*}}
\newcommand{\ee} {\end{eqnarray*}}

\newcommand{\dotp}[2]{\left\langle #1, #2\right\rangle}

\newcommand{\abs}[1]{\left\vert#1\right\vert}

\newcommand{\norm}[1]{\left\Vert#1\right\Vert}

\newcommand \bbP{\mathbb{P}}
\newcommand \bbE{\mathbb{E }}

\newcommand \bbR{\mathbb{R}}

\def\T{{ \mathrm{\scriptscriptstyle T} }}

\def\m{\mathcal}
\def\mb{\mathbb}

\def\mx{\mbox}

\newtheorem{theorem}{Theorem}[section]
\newtheorem{lemma}[theorem]{Lemma}

\newtheorem{corollary}[theorem]{Corollary}

\title{Bayesian fractional posteriors}

\author[1]{Anirban Bhattacharya\thanks{anirbanb@stat.tamu.edu}}
\author[2]{Debdeep Pati\thanks{debdeep@stat.fsu.edu}}
\author[2]{Yun Yang \thanks{yyang@stat.fsu.edu (Note: Authors are arranged in alphabetical order)}}
\affil[1]{Department of Statistics, Texas A\&M University}
\affil[2]{Department of Statistics, Florida State University}

\begin{document}
\maketitle
\begin{abstract}
\noindent We consider the {\em fractional posterior} distribution that is obtained by updating a prior distribution via Bayes theorem with a {\em fractional likelihood} function, a usual likelihood function raised to a fractional power. First, we analyze the contraction property of the fractional posterior in a general misspecified framework. 
Our contraction results only require a prior mass condition on certain Kullback-Leibler (KL) neighborhood of the true parameter (or the KL divergence minimizer in the misspecified case), and obviate constructions of test functions and sieves commonly used in the literature for analyzing the contraction property of a regular posterior. We show through a counterexample that some condition controlling the complexity of the parameter space is necessary for the regular posterior to contract, rendering additional flexibility on the choice of the prior for the fractional posterior. 
Second, we derive a novel Bayesian oracle inequality based on a PAC-Bayes inequality in misspecified models. Our derivation reveals several advantages of averaging based Bayesian procedures over optimization based frequentist procedures.
As an application of the Bayesian oracle inequality, we derive a sharp oracle inequality in the convex regression problem under an arbitrary dimension. We also illustrate the theory in Gaussian process regression and density estimation problems.
\end{abstract}

\section{Introduction}

The usage of {\it fractional likelihoods} has generated renewed attention in Bayesian statistics in recent years, where one raises a likelihood function to a fractional power, and combines the resulting fractional likelihood with a prior distribution via the usual Bayes formula to arrive at a {\em power posterior} or {\em fractional posterior} distribution. Applications of fractional posteriors has been diverse, ranging from fractional Bayes factors in objective Bayesian model selection \cite{o1995fractional}, data-dependent priors for sparse estimation \cite{martin2014asymptotically,martin2014empirical}, to marginal likelihood approximation \cite{friel2008marginal} and posterior simulation \cite{gelman1998simulating}. The fractional posteriors are a special instance of {\it Gibbs posteriors} \cite{jiang2008gibbs} or {\em quasi-posteriors} \cite{chernozhukov2003mcmc}, where the negative exponent of a loss function targeted towards a specific parameter of interest is used as a surrogate for the likelihood function; see \cite{bissiri2016general} for a general framework for updating of prior beliefs using Gibbs posteriors. 

The recent surge of interest in fractional posteriors can be largely attributed to its empirically demonstrated robustness to misspecification \cite{grunwald2012safe,miller2015robust}. For correctly specified, or well-specified (non)parametric models, there is now a rich body of literature \cite{ghosal2000,shen2001rates,ghosal2007convergence} guaranteeing concentration of the posterior distribution around minimax neighborhoods of the true data generating distribution. While it can be argued that the primary objective of Bayesian nonparametric models is to relax parametric assumptions to capture finer aspects of the data,  susceptibility to model misspecification remains a potent concern. First, in many practical situations, it may be unreasonable to assume that all aspects of data generating distribution can be captured adequately via a probabilistic model, and fitting models of increasingly high complexity additionally carries the risk of overfitting. Second, even though Bayesian nonparametric models can be arbitrarily flexible, they still rely on parametric building blocks, for example, component specific distributions in mixture models, and small perturbations to these assumptions can lead to fairly drastic differences in the inference drawn from the model fit \cite{miller2015robust}. 

There is a comparatively smaller literature on large sample behavior of nonparametric Bayesian procedures under misspecification \cite{kleijn2006misspecification,de2013bayesian,ramamoorthi2015posterior}, where the general aim is to establish sufficient conditions under which the usual posterior distribution concentrates around the nearest Kullback--Leibler (KL) point to the truth inside the parameter space. However, these conditions are considerably more stringent than those in case of well-specified models, so that verification can be fairly nontrivial, along with comparatively limited scope of applicability. In fact, \cite{grunwald2014inconsistency} empirically demonstrate through a detailed simulation study that even convergence to the nearest KL point may not take place in misspecified models. They instead recommend using a fractional posterior, with a data-driven approach to choose the fractional power; see also \cite{grunwald2012safe}. More recently, \cite{miller2015robust} proposed a coarsened posterior approach to combat model misspecification, where one conditions on neighborhoods of the empirical distribution rather than on the observed data to apply Bayes formula. When the neighborhood of the empirical distribution is chosen based on the KL divergence, \cite{miller2015robust} exhibited that the resulting coarsened posterior essentially takes the form of a fractional posterior. 
%Once again, the coarsened posterior exhibited favorable behavior under misspecification. 

These observations compel us to systematically study the concentration properties of fractional posteriors. While \cite{walker2001bayesian} established consistency of power posteriors for well-specified models; see also \cite{martin2016optimal} for certain rate results; we derive rates of convergence for the fractional posterior for general non-i.i.d. models in a misspecified model framework. The sufficient conditions for the fractional posterior to concentrate at the nearest KL point turn out to be substantially simpler compared to the existing literature on misspecified models. We state our concentration results for a class of R{\'e}nyi divergence measures in a non-asymptotic environment, which in particular, imply Hellinger concentration in properly specified settings. The effect of flattening the likelihood shows up in the leading constant in the rate. The sub-exponential nature of the posterior tails allow us to additionally derive posterior moment bounds. 
%Our general result applied to well-specified models implies that the fractional posterior achieves the same rate of concentration as the usual posterior in the Hellinger metric, though one pays a price in terms of the leading constant, which is not surprising. In fact, we state our concentration results for a class of Renyi divergence measures which imply Hellinger convergence. We comment here that all our concentration results are derived in a non-asymptotic environment.  

As one of our contributions, we show that the contraction rate of the fractional posterior is entirely determined by the prior mass assigned to appropriate KL neighborhoods of the true distribution, bypassing the construction of sieves and testing arguments in the existing theory \cite{barron1999,ghosal2000,kleijn2006misspecification}. One practically important consequence is that concentration results can be established for the fractional posterior for a much broader class of priors compared to the regular posterior. We provide several examples on usage of heavy tailed hyperpriors in density estimation and regression, where the fractional posterior provably concentrates at a (near) minimax rate, while the regular posterior has inconclusive behavior. Another novel application of our result lies in shape constrained function estimation. Obtaining metric entropy estimates in such problems pose a stiff technical challenge and constitutes an active area of research \cite{guntuboyina2013covering}. The fractional posterior obviates the need to obtain such entropy estimates en route to deriving concentration bounds. 

As a second contribution, we develop oracle inequalities for the fractional posterior based on a new PAC-Bayes inequality~\cite{catoni2004statistical,catoni2007,shawe1997pac,mcallester1998some,Dalalyan2008,guedj2013pac}  in a fully general Bayesian model. Many previous results on PAC-Bayes type inequalities are specifically tailored to classification (bounded loss, \cite{catoni2004statistical,catoni2007,shawe1997pac}) or regression (squared loss, \cite{shawe1997pac,leung2006information,Dalalyan2008,guedj2013pac}) problems.  Moreover, in the machine learning literature, a PAC-Bayes inequality is primarily used as a computational tool for controlling the generalization error by optimizing its upper bound over a restricted class of ``posterior" distributions~\cite{catoni2004statistical,catoni2007}. There is a need to develop a general PAC-Bayes inequality and an accompanied general theory for analyzing the Bayesian risk that can be applied to a broader class of statistical problems.
%The PAC-Bayes inequality developed in this paper aims to fill this gap. 
%There is also a deep connection between our investigations of the fractional posterior and the PAC-Bayesian literature \cite{shawe1997pac,mcallester1998some,Dalalyan2008,guedj2013pac}, where the usage of fractional likelihoods has been advocated in Gaussian-like models. In a Gaussian regression problem, \cite{leung2006information} proposed aggregating least squares estimators from various sub-models using a weighted average, where the weights can be interpreted as the posterior model probability under a flat prior on the model specific regression coefficients combined with a fractional likelihood. \cite{dalalyan2007aggregation} extended the formulation to the nonparametric regression context and derived an unbiased risk estimator for the risk of the posterior mean of a fractional posterior arising from a general class of likelihood functions which contained the Gaussian model; see also \cite{Dalalyan2008}. Our results extend and generalize these results in various directions as discussed subsequently in Section ??.  
In this paper, we derive an oracle type inequality for Bayesian procedures, which will be referred to as a {\em Bayesian oracle inequality} (BOI), based on a new PAC-Bayes inequality. Similar to the local Rademacher complexity \cite{bartlett2005} or local Gaussian complexity \cite{Bartlett:2003} in a frequenstist oracle inequality (FOI) for penalized empirical risk minimization procedures \cite{koltchinskii2011oracle,koltchinskii2006}, a BOI also involves a penalty term, which we refer to as {\em local Bayesian complexity}, that characterizes the local complexity of the parameter space. Roughly speaking, the local Bayesian complexity is defined as the inverse sample size times the negative logarithm of the prior mass assigned to certain Kullback-Leibler neighbourhood around the (pseudo) true parameter. In the special case when the prior distribution is close to be ``uniform'' over the parameter space, the local Bayesian complexity becomes the inverse sample size times a local covering entropy, and our BOI recovers the convergence rates derived from local covering conditions \cite{lecam1973}. Moreover, our BOI naturally leads to {\em sharp} oracle inequalities when the model is misspecified. For example, when applied to convex regression, we derive a sharp oracle inequality with minimax-optimal (up to $\log n$ factors) excess risk bound that extends the recent sharp oracle inequality obtained in \cite{Bellec:2015} from dimension one to general dimension $d\geq 1$.  

Last but not the least, our analysis reveals several potential advantages of averaging based Bayesian procedures over optimization based frequentist procedures.
First, due to the averaging nature of a Bayesian procedure, our averaging case analysis leading to a BOI is significantly simpler than a common worst case analysis leading to a FOI. For example, a local average type excess risk bound from a Bayesian procedure allows us to use simple probability tools, such as the Markov inequality and Chebyshev's inequality, to obtain a high probability bound for the excess risk, since the expectation operation exchanges with the averaging (integration) operation. This is different from a local supremum type excess risk from a optimization procedure, where more sophisticated empirical process tools are exploited to obtain a high probability bound for excess risk \cite{lugosi2004,bartlett2005,VandeGeer:00,van1996weak}, due to the non-exchangeability between the expectation operation and the supremum operation. 
For further details about the comparison between BOI and FOI, please refer to Section~\ref{subsec:PAC-Bayes}.
Second, a Bayesian procedure naturally leads to adaptation to unknown hyperparameters or tuning parameters. We show that by placing a hyper-prior that distributes proper weights to different levels of the hyperparameter, a BOI adaptively leads to the optimal rate corresponding to the best choice of the hyperparameter.
% In contrast, most common ways to select tuning parameters in a frequentist rely on cross-validation or data-splitting. These procedures use some proportion of the data for parameter estimation after learning the tuning parameter via the rest, and do not efficiently use all the data. \tcr{Delete the last line?}

The rest of the paper is organized as follows. 
%Notation used throughout the article are introduced in \S \ref{sec:prelim}. 
The main results of the paper are stated in \S \ref{sec:main}, with contraction results in \S \ref{subsec:conc}, and the PAC-Bayesian inequality and Bayesian oracle inequality in \S \ref{subsec:PAC-Bayes}. Applications to convex regression, Gaussian process regression and density estimation are discussed in \S \ref{sec:exp}. All proofs are deferred to \S\ref {section:proof}.

\section{Preliminaries}\label{sec:prelim}
We begin by introducing notation, and then briefly review R{\'e}nyi divergences as our key metric characterizing the contraction of fraction posteriors.

\subsection{Notation}
Let $C[0, 1]^d$ and $C^{\alpha}[0, 1]^d$ denote the space of continuous functions and the H\"{o}lder space of $\alpha$-smooth functions $f: [0, 1]^d \to \mathbb{R}$, respectively, endowed with the supremum norm $\norm{f}_{\infty} = \sup_{ t \in [0, 1]^d } \abs{ f(t) }$. For $\alpha > 0$, the H\"{o}lder space $C^{\alpha}[0, 1]^d$ consists of functions $f \in C[0, 1]^d$ that have bounded mixed partial derivatives up to order  $\lfloor \alpha \rfloor$, with the partial derivatives of order $\lfloor \alpha \rfloor$ being Lipschitz continuous of order $\alpha - \lfloor \alpha \rfloor$. Let $\norm{\cdot}_1$ and $\norm{\cdot}_2$ respectively denote the $L_1$ and $L_2$ norm on $[0, 1]^d$ with respect to the Lebesgue measure (i.e., the uniform distribution). To distinguish the $L_2$ norm with respect to the Lebesgue measure on $\mb R^d$, we use the notation $\norm{\cdot}_{2, d}$.  

%Let us denote the largest integer that is strictly smaller than $\beta$ by 
%$\lfloor \beta \rfloor$.  Let $\mathcal{Y} \subset \mathbb{R}^d$.  
%For $L:\mathcal{Y}\rightarrow[0,\infty)$, $\tau_0 \geq 0$, and $\beta>0$,
%a class of locally H\"older functions, $\mathcal{C}^{\beta,L,\tau_0}$, consists of 
%$f:\mathbb{R}^d\rightarrow \mathbb{R}$ such that for $k=(k_1,\ldots,k_d)$, $k_1+\cdots+k_d \leq \lfloor \beta \rfloor$, mixed partial derivative of order $k$, $D^k f$, is finite and for $k_1+\cdots+k_d = \lfloor \beta \rfloor$ and $\Delta y \in \mathcal{Y}$,
%\[
%|D^k f (y+\Delta y) - D^k f (y)| \leq L(y) ,\Delta y,^{\beta-\lfloor \beta \rfloor} e^{\tau_0 ,\Delta y,^2}.
%\]
For a finite set $A$, let $|A|$ denote the cardinality of $A$.
The set of natural numbers is denoted by  
$\mathbb{N}$. $a \lesssim b$ denotes $a \le C b$ for some constant $C > 0$. $J(\epsilon, A, \rho)$ denotes the $\epsilon$-covering number of the set $A$ with respect to the metric $\rho$. 
The $m$-dimensional simplex is denoted by $\Delta^{m-1}$. $I_k$ stands for the $k \times k$ identity matrix. Let $\phi_{\mu,\sigma}$ denote a multivariate normal density with mean $\mu \in \mathbb{R}^k$ and covariance 
matrix $\sigma^2 I_k$ (or a diagonal matrix with squared elements of $\sigma$ on the diagonal,
when $\sigma$ is a $k$-vector).

\subsection{ R{\'e}nyi divergences}\label{sec:renyi}
Let $P$ and $Q$ be probability measures on a common probability space with a dominating measure $\mu$, and let $p = dP/d\mu, q = dQ/d \mu$. The Hellinger distance $h^2(p, q) = (1/2) \int (\sqrt{p} - \sqrt{q})^2 d \mu = 1 - A(p, q)$, where $A(p, q) = \int \sqrt{p q} \, d \mu$ denotes the Hellinger affinity. Let $D(p, q) = \int p \log(p/q) d\mu$ denote the Kullback--Leibler (KL) divergence between $p$ and $q$. For any $\alpha \in (0, 1)$, let
\begin{align}\label{eq:renyi_def}
D_{\alpha}(p, q) = \frac{1}{\alpha-1} \log \int p^{\alpha} q^{1 - \alpha} d\mu 
\end{align}
denote the R{\'e}nyi divergence of order $\alpha$. Let us also denote $A_{\alpha}(p, q) = \int p^{\alpha} q^{1 - \alpha} d\mu = e^{-(1 - \alpha) D_{\alpha}(p, q)}$, which we shall refer to as the $\alpha$-affinity. When $\alpha = 1/2$, the $\alpha$-affinity equals the Hellinger affinity.  We recall some important inequalities relating the above quantities; additional details and proofs can be found in \cite{van2014renyi}. 

\noindent {\bf (R1)} $0 \le A_{\alpha}(p, q) \le 1$ for any $\alpha \in (0, 1)$, which in particular implies that $D_{\alpha}(p, q) \ge 0$ for any $\alpha \in (0, 1)$. 

\noindent {\bf (R2)} $D_{1/2}(p, q) = -2 \log A(p, q) = -2 \log \{ 1 - h^2(p, q) \} \ge 2 h^2(p, q)$ using the inequality $\log(1 + t) < t$ for $t > -1$. 

\noindent {\bf (R3)} For fixed $p, q$, $D_{\alpha}(p, q)$ is increasing in the order $\alpha \in (0, 1)$. Moreover, the following two-sided inequality shows the equivalence of $D_{\alpha}$ and $D_{\beta}$ for $0 < \alpha \le \beta < 1$:
$$
\frac{\alpha}{\beta} \frac{1 - \beta}{1 - \alpha} D_{\beta} \le D_{\alpha} \le D_{\beta}, \quad 0 < \alpha \le \beta < 1. 
$$
%To see this, let $\alpha \le \beta \in (0, 1)$. The function $x \mapsto x^{(\alpha-1)/(\beta - 1)}$ is strictly convex. Then, by Jensen's inequality, 
%\begin{align*}
%\frac{1}{\alpha-1} \log \int p^{\alpha} q^{1-\alpha} d \mu
%& = \frac{1}{\alpha-1}  \log \int \bigg( \frac{q}{p} \bigg)^{(1 - \beta) \frac{\alpha - 1}{\beta - 1}} dP \\
%&\le \frac{1}{\beta-1} \log \int \bigg( \frac{q}{p} \bigg)^{(1 - \beta)} dP.
%\end{align*}

\noindent {\bf (R4)} By an application of L'Hospital's rule, $\lim_{\alpha \to 1_{-}} D_{\alpha}(p, q) = D(p, q)$.

\section{Contraction and Bayesian oracle inequalities for fractional posteriors}\label{sec:main}
In this section, we present our main results. To begin with, we introduce the background including the definition of fractional posterior distributions in Bayesian procedures in \S\ref{section:background}. Then we present our results on the contraction of fraction posterior distributions in \S\ref{subsec:conc}, and Bayesian oracle inequalities based on PAC-Bayes type bounds in \S\ref{subsec:PAC-Bayes}.

\subsection{Background}\label{section:background}
We will present our theory on the large sample properties of fractional posteriors in its full generality by allowing the model to be misspecified and the observations, denoted by $X^{(n)}=(X_1,X_2,\ldots,X_n)$, to be neither identically nor independently distributed (abbreviated as non-i.i.d.) \cite{ghosal2007convergence}. Our non-i.i.d. result can be applied to models with nonindependent observations such as Gaussian time series and Markov processes, or models with independent, nonidentically distributed (i.n.i.d.) observations such as Gaussian regression and density regression.

More specifically, let $(\mathcal{X}^{(n)},\mathcal{A}^{(n)}, \bbP^{(n)}_\theta:\, \theta\in \Theta)$ be a sequence of statistical experiments with observations $X^{(n)}$, where $\theta$ is the parameter of interest in arbitrary parameter space $\Theta$, and $n$ is the sample size. For each $\theta$, let $\bbP^{(n)}_\theta$ admit a density $p_{\theta}^{(n)}$ relative to a $\sigma$-finite measure $\mu^{(n)}$. Assume that $(x,\theta)\to p_{\theta}^{(n)}(x)$ is jointly measurable relative to $\mathcal{A}^{(n)} \otimes\mathcal{B}$, where $\mathcal{B}$ is a $\sigma$-field on $\Theta$. We place a prior distribution $\Pi_n$ on $\theta\in\Theta$, and define the fractional likelihood of order $\alpha \in (0, 1)$ to be the usual likelihood raised to power $\alpha$:
\begin{align}\label{eq:general_frac_lik}
L_{n, \alpha}(\theta) = \bigg[ p^{(n)}_\theta (X^{(n)}) \bigg]^{\alpha}. 
\end{align}
Let $\Pi_{n, \alpha}(\cdot)$ denote the posterior distribution obtained by combining the fractional likelihood $L_{n, \alpha}$ with the prior $\Pi_n$, that is, for any measurable set $B \in\mathcal{B}$, 
\begin{align}\label{eq:frac_post}
\Pi_{n, \alpha}( B\,| X^{(n)}) = \frac{ \int_{B} L_{n, \alpha}(\theta)\,  \Pi_n(d\theta)}{ \int_{\Theta} L_{n, \alpha}(\theta) \, \Pi_n(d\theta)} = \frac{\int_B e^{-\alpha\, r_{n}(\theta,\, \theta^\dagger)}\, \Pi_n(d\theta)}{\int_{\Theta} e^{-\alpha\, r_{n}(\theta,\, \theta^\dagger)}\, \Pi_n(d\theta)},
\end{align} 
where $r_{n}(\theta,\theta^\dagger):\, = \log \{p^{(n)}_{\theta^\dagger} (X^{(n)})/p^{(n)}_\theta (X^{(n)})\}$ is the negative log-likelihood ratio between $\theta$ and any other fixed parameter value $\theta^\dagger$. For example, we may choose $\theta^\dagger$ as the parameter $\theta_0$ associated with the true data generating distribution, abbreviated as the true parameter. Clearly, $\Pi_{n,1}$ denotes the usual posterior distribution. 

We allow the model to be misspecified by allowing $\theta_0$ to lie outside the parameter space $\Theta$. In misspecified models, the point $\theta^\ast$ in $\Theta$ that minimizes the KL divergence from $\bbP_{\theta_0}^{(n)}$, that is,
\begin{align}\label{eq:KL_min}
\theta^\ast :\,=\arg\min_{\theta\in\Theta} D\big( p_{\theta_0}^{(n)}, p_{\theta}^{(n)}\big),
\end{align}
plays the role of $\theta_0$ in well-specified models \cite{kleijn2006misspecification}. In fact, we will show that the fractional posterior distribution $\Pi_{n, \alpha}$ tends to contract towards $\theta^\ast$ as $n\to\infty$. We use the divergence 
\begin{align}\label{eq:alpha_div}
D^{(n)}_{\theta_0, \alpha}(\theta, \theta^\ast):\,= \frac{1}{\alpha -1}\log A^{(n)}_{\theta_0,\alpha}(\theta, \theta^\ast),
\end{align}
referred to as the $\alpha$-divergence with respect to $\bbP^{(n)}_{\theta_0}$, or simply $\theta_0$, to measure the closeness between any $\theta\in\Theta$ and $\theta^\ast$, where 
\begin{align*}
A^{(n)}_{\theta_0,\alpha}(\theta, \theta^\ast):\,= \int \left(\frac{p^{(n)}_{\theta}}{p^{(n)}_{\theta^\ast}}\right)^\alpha p^{(n)}_{\theta_0} \, d\mu^{(n)}
\end{align*}
is an $\alpha$-affinity between $\theta$ and $\theta^\ast$ with respect to $\theta_0$.

\noindent {\bf Remark:} In the well-specified case where $\theta^\ast =\theta_0\in\Theta$, $A^{(n)}_{\theta_0,\alpha}$ reduces to the usual $\alpha$-affinity defined in \S \ref{sec:renyi}, 
and $D^{(n)}_{\theta_0,\alpha}$ becomes the R{\'e}nyi divergence of order $\alpha$ between $p^{(n)}_{\theta}$ and $p^{(n)}_{\theta_0}$:
\begin{align}\label{Eqn:Renyi}
D^{(n)}_{\alpha}(\theta, \theta_0) = D_{\alpha}\big(p^{(n)}_{\theta}, p^{(n)}_{\theta_0}\big) = \frac{1}{\alpha-1} \log \int \{p^{(n)}_\theta\}^{\alpha} 
\{p^{(n)}_{\theta_0}\}^{1 - \alpha} d\mu^{(n)}.
\end{align}
Note we drop $\theta_0$ from the subscript when $\theta^* = \theta_0$. 

\

In general, $D^{(n)}_{\theta_0,\alpha}$ continues to define a divergence measure that satisfies $D^{(n)}_{\theta_0, \alpha}(\theta,\theta^\ast) \geq 0$ for $\theta\in\Theta$ and $D^{(n)}_{\theta_0, \alpha}(\theta^\ast,\theta^\ast)=0$ in a variety of statistical problems. For example, in the normal means problem $Y \sim N(\mu, \sigma^2 I_n)$, $D^{(n)}_{\theta_0,\alpha}$ defines a divergence measure if the parameter space for the mean $\mu$ is a convex set in $\mb R^n$; see \S \ref{Section:NPregression} for more details. The convexity condition is satisfied by a broad class of problems, including sparse problems, isotonic regression, and convex regression \cite{chatterjee2014new}. In the density estimation context, the following Lemma shows that $D^{(n)}_{\theta_0,\alpha}$ defines a divergence measure if the parameter space of densities is convex. 
 
\begin{lemma}[Property of $\alpha$-divergences]\label{Lem:divergence}
If $\{ p_{\theta}^{(n)} : \theta \in \Theta\}$ is convex\footnote{Given any $\theta, \theta' \in \Theta$, and $\omega \in (0, 1)$, there exists $\bar{\theta} \in \Theta$ such that $p_{\bar{\theta}}^{(n)} = \omega p_{\theta}^{(n)} + (1 - \omega) p_{\theta'}^{(n)}$} or $\theta^*$ is an interior point of $\Theta$, then $0<A^{(n)}_{\theta_0,\alpha}(\theta,\theta^\ast) \leq 1$ for any $\alpha\in(0,1)$. Therefore, $D^{(n)}_{\theta_0,\alpha}$ defines a divergence that satisfies $D^{(n)}_{\theta_0, \alpha}(\theta,\theta^\ast) \geq 0$ for $\theta\in\Theta$ and $D^{(n)}_{\theta_0, \alpha}(\theta^\ast,\theta^\ast)=0$.
\end{lemma}
When $\alpha\in(0,1)$, the proof of the lemma implies that $D^{(n)}_{\theta_0,\alpha}(\theta,\theta^\ast)=0$ if and only if $p^{(n)}_{\theta}=p^{(n)}_{\theta^\ast}$ on the support of $\bbP^{(n)}_{\theta_0}$, since $x^\alpha$ is a strictly concave function on $[0,\infty)$.

\

\noindent We will primarily focus on the following two cases in this paper. 

\paragraph{Independent and identically distributed observations:} When $X_1,X_2,\ldots, X_n$ are i.i.d.~observations, $\bbP_{\theta}^{(n)}$ equals the $n$-fold product measure $\bbP^n_{\theta}:\,=\otimes_{i=1}^n \bbP_{\theta}$, where $\bbP_{\theta}$ is the common distribution for the observations. $\mathcal{A}^{(n)}$ also takes a product form as $\mathcal{A}^n:\,= \otimes_{i=1}^n \mathcal{A}$, with $\mathcal{A}$ the common $\sigma$-field. The fractional likelihood function is
\begin{align}\label{eq:iid_frac_lik}
L_{n, \alpha}(\theta) = \prod_{i=1}^n \big\{ p_{\theta}(X_i) \big\}^{\alpha},
\end{align}
where $p_{\theta}$ is the common density indexed by $\theta\in\Theta$. The negative log-likelihood ratio $r_n(\theta,\theta^\dagger) = \sum_{i=1}^n  \log \{p_{\theta^\dagger} (X_i)/p_\theta (X_i)\}$ becomes the sum of individual log density ratios. Moreover, the $\alpha$-affinity and divergence can be simplified as $A^{(n)}_{\theta_0,\alpha}(\theta, \theta^\ast) = \big\{A_{\theta_0,\alpha}(\theta, \theta^\ast)\big\}^n$ and $D^{(n)}_{\theta_0,\alpha}(\theta, \theta^\ast) = n\, D_{\theta_0,\alpha}(\theta, \theta^\ast)$, where $A_{\theta_0,\alpha}(\theta, \theta^\ast)$ and $D_{\theta_0,\alpha}(\theta, \theta^\ast)$ respectively are the $\alpha$-affinity and divergence for $n=1$.

\paragraph{Independent observations:} In this case as well, $X_1,X_2,\ldots, X_n$ are independent observations. However, the $i$th observation $X_i$ has an index-dependent distribution
$\bbP_{\theta,i}$, which possesses a density $p_{\theta,i}$ relative to a $\sigma$-finite measure $\mu_i$ on $(\mathcal{X}_i,\mathcal{A}_i)$. Thus, we take the measure $\bbP_{\theta}^{(n)}$ equal to the product measure $\otimes_{i=1}^n \bbP_{\theta,i}$ on the product measurable space $\otimes_{i=1}^n (\mathcal{X}_i,\mathcal{A}_i)$.
The fractional likelihood function takes a product form as
\begin{align}\label{eq:iid_frac_lik}
L_{n, \alpha}(\theta) = \prod_{i=1}^n \big\{ p_{\theta,i}(X_i) \big\}^{\alpha},
\end{align}
and the negative log-likelihood ratio $r_n(\theta,\theta^\dagger) = \sum_{i=1}^n  \log \{p_{\theta^\dagger,i} (X_i)/p_{\theta,i} (X_i)\}$. The $\alpha$-affinity and divergence can be decomposed, respectively, as $A^{(n)}_{\theta_0,\alpha}(\theta, \theta^\ast) = \prod_{i=1}^n A_{\theta_0,\alpha, i}(\theta, \theta^\ast)$ and $D^{(n)}_{\theta_0,\alpha}(\theta, \theta^\ast) = \sum_{i=1}^n D_{\theta_0,\alpha, i}(\theta, \theta^\ast)$, where $A_{\theta_0,\alpha,i}(\theta, \theta^\ast)$ and $D_{\theta_0,\alpha,i}(\theta, \theta^\ast)$ are the $\alpha$-affinity and divergence associated with the $i$th observation $X_i$.

\subsection{General concentration bounds}\label{subsec:conc}
In this subsection, we consider the asymptotic behavior of fractional posterior distributions and corresponding Bayes estimators based on non-i.i.d.~observations under the general misspecified framework. We give general results on the rate of contraction of the fractional posterior measure towards the KL minimizer $\theta^\ast$ relative to the $\alpha$-divergence $D^{(n)}_{\theta_0,\alpha}$.

For any $\theta^\dagger$, we define a specific kind of KL neighborhood of $\theta^\dagger$ with radius $\varepsilon$ as 
\begin{align}\label{eq:chEN}
B_{n}(\theta^\dagger,\varepsilon; \theta_0) = \bigg\{\theta\in\Theta:  \int p_{\theta_0}^{(n)} \log (p_{\theta^\dagger}^{(n)}/p_{\theta}^{(n)}) d\mu^{(n)} \leq n\varepsilon^2, \,
\int p_{\theta_0}^{(n)} \log^2 (p_{\theta^\dagger}^{(n)}/p_{\theta}^{(n)})  d\mu^{(n)}\leq n\varepsilon^2  \bigg\}.
\end{align}
It is standard practice to make assumptions on the prior mass assigned to such KL neighborhoods to obtain the rate of posterior concentration in misspecified models \cite{kleijn2006misspecification}. 
With these notations, we present a {\em nonasymptotic} upper bound for the posterior probability assigned to complements of $\alpha$-divergence neighborhoods of $\theta^\ast$ with respect to $\theta_0$. 

\begin{theorem}[Contraction of fractional posterior distributions] \label{Thm:contraction}
Fix $\alpha\in(0,1)$. Recall $\theta^*$ from \eqref{eq:KL_min}. Assume that $\varepsilon_n$ satisfies $n\, \varepsilon_n^2 \geq 2$ and 
\begin{align}\label{eq:prior_Mass}
\Pi_n\big(B_{n}(\theta^\ast,\varepsilon_n; \theta_0)\big) \geq e^{-n\,\varepsilon_n^2}. % \quad\mbox{and}\quad n\, \varepsilon_n^2 \geq 2. 
\end{align}
Then, for any $D\geq 2$ and $t>0$, 
\begin{align*}
\Pi_{n,\alpha} \Big(\frac{1}{n}D^{(n)}_{\theta_0,\alpha}(\theta,\,\theta^\ast) \geq \frac{D+3t}{1-\alpha}\, \varepsilon_n^2\ \Big| \, X^{(n)}\Big) \leq e^{-t\, n\, \varepsilon_n^2}
\end{align*}
holds with $\bbP^{(n)}_{\theta_0}$ probability at least $1 - 2/\{(D-1+t)^2 n \varepsilon_n^2\}$.
\end{theorem}
Theorem~\ref{Thm:contraction} characterizes the contraction of the fractional posterior measure where the posterior of $D^{(n)}_{\theta_0,\alpha}(\theta,\,\theta^\ast)$ exhibits a sub-exponentially decaying tail. As a direct consequence, we have the following corollary that characterizes the fractional posterior moments of $D^{(n)}_{\theta_0,\alpha}(\theta,\,\theta^\ast)$. 

\begin{corollary}[Fractional posterior moments]\label{Cor:contraction}
Under the conditions of Theorem~\ref{Thm:contraction}, we have that for any $k\geq 1$, 
\begin{align*}
\int \Big\{\frac{1}{n}D^{(n)}_{\theta_0,\alpha}(\theta,\,\theta^\ast)\Big\}^k\,  \Pi_{n,\alpha} (d\theta\, \big| \, X^{(n)})  \leq  \frac{C_1}{(1-\alpha)^k}\,\varepsilon_n^{2k},
\end{align*}
holds with $\bbP^{(n)}_{\theta_0}$ probability at least $1 - C_2/\{n \varepsilon_n^2\}$, where $(C_1,C_2)$ are some positive constants depending on $k$. 
\end{corollary}

\,

\noindent {\bf Implications for well-specified models.} While Theorem \ref{Thm:contraction} and Corollary \ref{Cor:contraction} apply generally to the misspecified setting, it is instructive to first consider their implications in the well-specified setting, i.e., when the data generating parameter $\theta_0 \in \Theta$. Setting $t = 1$ in Theorem \ref{Thm:contraction} implies that the fractional posterior increasingly concentrates on $\varepsilon_n$-sized $D_{\theta_0, \alpha}^{(n)}$ neighborhoods of the true parameter $\theta_0$. In particular, given {\bf (R2)} and {\bf (R3)}, Theorem \ref{Thm:contraction} implies that for any $\alpha \in (0, 1)$, the rate of concentration of the fractional posterior $\Pi_{n, \alpha}$ in the Hellinger metric is $\varepsilon_n$. Similar concentration results for the usual posterior distribution in the Hellinger metric were established in \cite{shen2001rates,ghosal2000} for the i.i.d. case, and in \cite{ghosal2007convergence} for the non-i.i.d. case. Since the prior mass condition \eqref{eq:prior_Mass} appears as one of the sufficient conditions there as well, the fractional posterior achieves the same rate of concentration as the usual posterior (albeit up to constants) in all the examples considered in these works, which is typically minimax up to a logarithmic term for appropriately chosen priors. 
%In the i.i.d. setting \cite{shen2001rates,ghosal2000}, the neighborhood 
%$$
%B_{n}(\theta_0,\varepsilon; \theta_0) = \bigg\{\theta\in\Theta:  \int p_{\theta_0} \log (p_{\theta_0}/p_{\theta}) d\mu \leq \varepsilon^2, \,
%\int p_{\theta_0} \log^2 (p_{\theta_0}/p_{\theta}) d\mu \leq \varepsilon^2  \bigg\}. 
%$$
In addition to the prior mass condition \eqref{eq:prior_Mass}, the sufficient conditions of \cite{ghosal2007convergence} additionally require the construction of {\em sieves} $\m F_n \subset \Theta$ whose $\varepsilon_n$-entropy in the Hellinger metric is stipulated to grow in the order $\lesssim e^{ C n \varepsilon_n^2}$, and at the same time, the prior probability assigned to the complement of the sieve is required to be exponentially small, i.e., $\Pi_n(\m F_n^c) \le e^{- C' n \varepsilon_n^2}$. The existence of such sieves with suitable control over their metric entropy is a crucial ingredient of their theory, as it guarantees existence of exponentially consistent test functions \cite{birge1984tests,le1986asymptotic} to test the true density against complements of Hellinger neighborhoods of the form $\{\theta \in \m F_n : h^2\big(p_{\theta}^{(n)}, p^{(n)}_{\theta_0}\big) \ge M \varepsilon_n^2\}$. 
%Construction of such point null versus composite hypothesis tests follow the celebrated Birge-Le Cam testing theory \cite{birge1984tests,le1986asymptotic}. 

An important distinction for the fractional posterior in Theorem \ref{Thm:contraction} is that the prior mass condition {\em alone} is sufficient to guarantee optimal concentration. This is important for at least two distinct reasons. First, the condition of exponentially decaying prior mass assigned to the complement of the sieve implies fairly strong restrictions on the prior tails and essentially rules out heavy-tailed prior distributions on hyperparameters. On the other hand, a much broader class of prior choices lead to provably optimal posterior behavior for the fractional posterior. Second, obtaining tight bounds on the metric entropy in non-regular parameter spaces, for example, in shape-constrained regression problems, can be a substantially nontrivial exercise \cite{guntuboyina2013covering}, which is entirely circumvented using the fractional posterior approach. Specific examples of either kind are provided in \S \ref{sec:exp}. 

While it may be argued that the conditions on the entropy and complement probability of the sieve are only sufficient conditions, a counterexample from \cite{barron1999} suggests that some control on the complexity of the parameter space is also necessary to ensure the consistency of a regular posterior when the model space is well-specified. Specifically, in their example, the posterior tends to put all its mass on a set of distributions that are $\sqrt{2-\sqrt{2}}$ away from the true data generating distribution with respect to the Hellinger metric, even though the prior assigns positive probability over any $\varepsilon$-KL ball around the true parameter. As an implication, the fractional posterior can still achieve a certain rate of contraction for this problem even though the regular posterior is not consistent. In fact, the rate of concentration of the fractional posterior $\varepsilon_n = (1-\alpha)^{-1} n^{-1/3}$ for this problem, since their prior satisfies $\Pi_n(B_{n}(\theta_0,\varepsilon; \theta_0)) \geq e^{-C\, \varepsilon^{-1}}$ for some constant $C>0$.  Therefore, a combination of Theorem~\ref{Thm:contraction} and the counterexample in \cite{barron1999} shows that the fractional posterior has an {\em annealing effect} that can flatten the potential peculiar spikes in the regular posterior that are far away from the true parameter. However, this additional flexibility of the fractional posterior comes at a price---when the regular posterior contracts, then the $\alpha$-fractional posterior will sacrifice a factor of $(1-\alpha)^{-1}$ in the rate of contraction. 

The following theorem shows that for fixed $n$, the fractional posterior will almost surely converges to the regular posterior ($\alpha=1$) as $\alpha \to 1_{-}$.
\begin{theorem}[Regular posterior as a limit of fractional posteriors]\label{Thm:fractional_convergence}
For each $n$, we have
\begin{align*}
\bbP^{(n)}_{\theta_0}\big[\Pi_{n,1}(B\,|\,X^{(n)}) = \lim_{\alpha\to 1_{-}} \Pi_{n,\alpha}(B\,|\,X^{(n)}), \, \forall B\in\mathcal{B}\big] = 1.
\end{align*}
\end{theorem}
This theorem implies that although for a fixed $\alpha\in(0,1)$, the fraction posterior has the annealing effect of flattening the posterior, it will eventually convergence to the regular posterior as $\alpha \to 1_{-}$ almost surely. This observation also justifies the empirical observation \cite{Geyer1995} that parallel tempering can boost the convergence of the posterior when the posterior contracts. However, when the posterior is ill-behaved---does not have consistency or has multimodality, then we need a very fine grid for the design of $\alpha$ as $\alpha\to 1_{-}$ in the parallel tempering algorithm, since otherwise all factional posteriors will only exhibit the one big mode around $\theta^\ast$ and miss the rest.

%\paragraph{Comparisons with previous work:}
%\textcolor{red}{To be edited}
%Defining $\{H_\alpha^*(f_0, f)\}^2 =  1- A_{\alpha}^*(f_0, f)$. If $\mathcal{F}$ is convex or $f^*$ is an interior point of $\mathcal{F}$,  then with $\bbP^{(n)}_{\theta_0}$ probability at least $(1 - \varepsilon)$,  
%\begin{align}\label{eq:bd3}
% \int  n\{H^*_{\alpha}(f_0, f)\}^2  \Pi_n(df) \le \inf_{\rho \ll \Pi} \bigg[\int r_{n,\alpha}^*(f) \rho(df) + D(\rho \vert \vert \Pi) \bigg] + \log(1/\varepsilon).
%\end{align}
%\emph{Comments on convexity},
\paragraph{Implications for misspecified models.}

A key reference for Bayesian asymptotics in infinite-dimensional misspecified models is \cite{kleijn2006misspecification}, where sufficient conditions analogous to the well-specified case were provided for the posterior to concentrate around $\theta^*$. The primary technical difficulty in showing such a result compared to the well-specified case is the construction of test functions, for which \cite{kleijn2006misspecification} proposed a novel solution. Akin to the well-specified case for the regular posterior, the sufficient conditions of \cite{kleijn2006misspecification} constitute of a prior thickness condition as in Theorem \ref{Thm:contraction}, and conditions on entropy numbers. However, the entropy number conditions (equations (2.2) and (2.5) in \cite{kleijn2006misspecification}) for the misspecified case are substantially harder to verify.  In their Lemma 2.1, a simpler sufficient condition related their entropy number condition to ordinary entropy numbers.  Further, in their Lemma 2.3, exploiting convexity of the parameter space, they established that the sufficient conditions of their Lemma 2.1 are satisfied by a weighted Hellinger distance 
\begin{eqnarray*}
h_w^2({\theta}^{(n)}, {\theta^*}^{(n)}) =  \frac{1}{4} \int \bigg(\sqrt{p_{\theta^*}^{(n)}} - \sqrt{p_{\theta}^{(n)}} \bigg)^2 \, \frac{p_{\theta_0}^{(n)}}
{p_{\theta^*}^{(n)}} ~d \mu^{(n)}, 
\end{eqnarray*}
which then amounts to obtaining entropy numbers in the weighted Hellinger metric. Such an exercise typically requires further assumptions on the behavior of $p_{\theta}^{(n)}/p_{\theta^*}^{(n)}$. For example, if $\sup_{\theta}  \abs{p_{\theta}^{(n)}/p_{\theta^*}^{(n)}}$ is finite, the ordinary Hellinger metric dominates the weighted Hellinger metric and it suffices to obtain covering numbers with respect to the ordinary Hellinger metric. Under this assumption, the authors proceeded to derive convergence rates for the regular posterior in a density estimation problem using Dirichlet process mixture priors. However, this assumption precludes the true density $p_{\theta_0}^{(n)}$ to have heavier tails than that prescribed by the model. For example, if the true density is heavier that the class of densities specified by the model,  the assumption $\sup \abs{p_{\theta}^{(n)}/p_{\theta^*}^{(n)}} < \infty$ is not satisfied.  Typically,  in the misspecified case, controlling the  prior  
mass \eqref{eq:prior_Mass} in Theorem 3.2 requires certain tail conditions on $f_0$.     However, Theorem \ref{Thm:contraction} obviates the need to verify any entropy conditions for the fractional posterior. It thus avoids the need to assume $\sup \abs{p_{\theta}^{(n)}/p_{\theta^*}^{(n)}} < \infty$, unless required to verify the prior mass condition.

For $\alpha = 1/2$, our divergence measure $D_{1/2}({\theta}^{(n)}, {\theta^*}^{(n)})$ dominates the weighted Hellinger distance in which \cite{kleijn2006misspecification}  derive their convergence rate for the density estimation problem in Theorem 3.1. This can be readily seen from 
\begin{eqnarray*}
4h_w^2({\theta}^{(n)}, {\theta^*}^{(n)})  &=& 
1 + \int \frac{p_{\theta}^{(n)}}{p_{\theta^*}^{(n)}} \, p_{\theta_0}^{(n)}d \mu^{(n)}
  - 2 \int \bigg( \frac{p_{\theta}^{(n)}}{ p_{\theta^*}^{(n)}} \bigg)^{1/2} p_{\theta_0}^{(n)} d \mu^{(n)} \\
  & \leq& 2\bigg[ 1 -   \int \bigg( \frac{p_{\theta}^{(n)}}{p_{\theta^*}^{(n)}} \bigg)^{1/2} p_{\theta_0}^{(n)}d \mu^{(n)} \bigg]  \le D_{1/2}({\theta}^{(n)}, {\theta^*}^{(n)}),
\end{eqnarray*}
where the last inequality follows from $\log x \leq x-1$ and the penultimate inequality follows from Lemma \ref{Lem:divergence}.
%Hence, using a fractional likelihood will automatically bypass the entropy conditions  making Theorem \ref{Thm:contraction} stronger and more broadly applicable. 

\subsection{PAC-Bayes bounds and Bayesian oracle inequalities}\label{subsec:PAC-Bayes}

In many problems, the performance of a (pseudo) Bayesian approach can be characterized via PAC-Bayes type inequalities \cite{shawe1997pac,mcallester1998some,guedj2013pac}. A typical PAC-Bayes inequality takes the form as
\begin{align*}
\int R(\theta,\, \theta_0)\, \Pi_{n,\alpha}(d\theta\,|\,X^{(n)}) \leq  \int S_n(\theta,\,\theta_0) \rho(d\theta) + \frac{1}{\kappa_n} D(\rho\,,\,\Pi_n)+ \mbox{Rem},\ \ \ \mbox{$\forall$ probability measure $\rho\ll \Pi_n$},
\end{align*}
where $R$ is a statistical risk function, $\kappa_n$ is some tuning parameter, $ \mbox{Rem}$ is some remainder term, and $S_n$ is some function that measures the discrepancy between $\theta$ and $\theta_0$ on the support of $\rho$.  We present a PAC-Bayes inequality for the fractional posterior distribution, where the risk function $R$ is a multiple of the $\alpha$-R{\'e}nyi divergence $D^{(n)}_\alpha$ in~\eqref{Eqn:Renyi}, and $S_n(\theta,\theta_0)$ a multiple of the negative log-likelihood ratio $r_n(\theta,\theta_0)$.

\begin{theorem}[PAC-Bayes inequality relative to $\theta_0$]\label{thm:renyi_main}
Fix $\alpha \in (0, 1)$. Then, for any $\varepsilon \in (0, 1)$, %it holds that
\begin{equation}\label{eq:renyi_bd}
\begin{aligned}
 \int \Big\{\frac{1}{n} D^{(n)}_{\alpha}(\theta, \theta_0) \Big\}\, \Pi_{n, \alpha}(d\theta\,|\,X^{(n)}) \le \frac{\alpha}{n (1 - \alpha)} \int r_{n}(\theta,\,\theta_0) \,\rho(d\theta) + &\frac{1}{n(1-\alpha)} D(\rho \,,\,\Pi_n)  + \frac{1}{n(1-\alpha)}  \log(1/\varepsilon),\\
  &\mbox{$\forall$ probability measure $\rho\ll \Pi$},
\end{aligned}
\end{equation}
with $\bbP^{(n)}_{\theta_0}$ probability at least $(1 - \varepsilon)$. 
\end{theorem}

Theorem~\ref{thm:renyi_main} immediately implies an oracle type inequality for the Bayes estimator $\widehat{\theta}_B:\,= \int_{\Theta} \theta \, \Pi_{n, \alpha}(d\theta\,|\,X^{(n)})$ by using the convexity of $D^{(n)}_{\alpha}(\cdot, \theta_0)$ and applying Jensen's inequality,
\begin{align}\label{eq:BOI}
 \frac{1}{n} D^{(n)}_{\alpha}(\hat{\theta}_B, \, \theta_0)\le \frac{\alpha}{n (1 - \alpha)} \int r_{n}(\theta,\,\theta_0) \,\rho(d\theta) + \frac{1}{n(1-\alpha)} D(\rho \,,\,\Pi_n)  &+ \frac{1}{n(1-\alpha)}  \log(1/\varepsilon),
\end{align}
for all probability measure $\rho\ll \Pi_n$. We call this inequality a {\em Bayesian oracle inequality}. 

Let us compare the Bayesian oracle inequality (BOI) with frequentist oracle inequalities (FOI) \cite{koltchinskii2011oracle,koltchinskii2006}. For convenience, we assume that the observations are i.i.d., and use $\bbP_n$ to represent the empirical measure $\frac{1}{n}\sum_{i=1}^n \delta_{X_i}$. For a function $f:\m X \to \bbR$, 
define
\begin{align}\label{eq:FOI}
\bbP_n f = \frac{1}{n}\sum_{i=1}^n f(X_i),\qquad\mbox{and}\qquad \bbP_{\theta_0} f= \bbE_{\theta_0} f(X).
\end{align}
Under this notation, a typical FOI takes a form as
\begin{align}\label{eq:FOI}
\bbP_{\theta_0} f_{\widehat{\theta}} \leq c \inf_{\theta\in\Theta} \bbP_{\theta_0} f_\theta + \Psi_n(r_n),
\end{align}
for some leading constant $c\geq 1$, where $\widehat{\theta}$ is the estimator of $\theta$, for example, obtained by empirical risk minimization \cite{koltchinskii2006,Bartlett2004}.
Here $\m F=\{f_{\theta}:\, \m X\to\bbR,\, \theta\in \Theta\}$ is a class of functions indexed by $\theta\in\Theta$, such as, a certain loss function $\ell(\cdot, X)$ evaluated at $\theta$. The term $\inf_{\theta\in\Theta} \bbP_{\theta_0} f_\theta$ will be referred to as the approximation error term,  reflecting the smallest loss incurred by approximating $f_{\theta_0}$ from $\m F$. The second term $\Psi_n(r_n)$ in the display is an excess risk term that reflects certain local complexity measure of $\m F$, such as the local Rademacher complexity \cite{bartlett2005} or local Gaussian complexity \cite{Bartlett:2003}. $\Psi_n(r_n)$ typically serves as a high probability upper bound to the supremum of the localized empirical process, 
\begin{align}\label{eq:FOI_LC}
 \sup_{\theta\in \Theta:\, \bbP_n f_{\theta} \leq r_n} \big\{ \bbP_n f_{\theta} - \bbP_{\theta_0} f_{\theta}\big\},
\end{align}
up to some other remainder terms, where $r_n$ is a critical radius obtained as the fixed point of certain function depending on $\Psi_n$. 

Now let us look at the BOI~\eqref{eq:BOI}, which can be rewritten as
\begin{equation}\label{eq:BOIb}
\begin{aligned}
 \frac{1}{n} D^{(n)}_{\alpha}(\hat{\theta}_B, \, \theta_0)\le  &\, \frac{\alpha}{n (1 - \alpha)} \inf_{\theta\in\Theta} \bbP_{\theta_0} r_\theta + \frac{\alpha}{n (1 - \alpha)} \int \big\{\bbP_n r_{\theta} - \bbP_{\theta_0} r_\theta\big\} \,\rho(d\theta) \\
 &+ \Big\{ \frac{\alpha}{n (1 - \alpha)}  \int \big\{\bbP_{\theta_0} r_{\theta} - \inf_{\theta\in\Theta}\bbP_{\theta_0} r_\theta\big\} \,\rho(d\theta) + \frac{1}{n(1-\alpha)} D(\rho \,,\,\Pi_n) + \frac{1}{n(1-\alpha)}  \log(1/\varepsilon)\Big\},
\end{aligned}
\end{equation}
where $r_{\theta}(X)=\log\{p_{\theta_0}(X)/p_{\theta}(X)\}$ is the log density ratio.
We observe that the first term on the right hand side of \eqref{eq:BOIb} is the approximation error term, and the rest serves as the excess risk term. However, the excess risk term in BOI has two distinctions from that in FOI. First, different from the FOI that induces localization via either an iterative procedure \cite{Koltchinskii2000} or solving the fixed point of certain function \cite{bartlett2005}, a BOI induces localization via picking a measure $\rho$ concentrating around the best approximation $\arg\min_{\theta\in\Theta} \bbP_{\theta_0} r_\theta$ that balances between the average approximation error $ \int \big\{\bbP_{\theta_0} r_{\theta} - \inf_{\theta\in\Theta}\bbP_{\theta_0} r_\theta\big\} \,\rho(d\theta) $ and a penalty on the size of localization $D(\rho \,,\,\Pi_n)$. Second, in FOI the stochastic term characterizing the local complexity is based on a worse case analysis
by taking the supremum as in \eqref{eq:FOI_LC}, while BOI bounds the stochastic term based on an average case analysis via the average fluctuation
\begin{align*}
\int \big\{\bbP_n r_{\theta} - \bbP_{\theta_0} r_\theta\big\} \,\rho(d\theta).
\end{align*}
Because we can exchange the expectation with integration, this local average form allows us to use simple probability tools, such as Markov's inequality and Chebyshev's inequality, to obtain bounds for the excess risk. This is different from the local supremum form~\eqref{eq:FOI_LC}, where expectation does not exchange with supremum, and we need much more sophisticated empirical process tools such as chaining and peeling techniques to bound the excess risk (see, for example, \cite{lugosi2004,bartlett2005,VandeGeer:00,van1996weak}). 

As a simple illustration of applying Chebyshev's inequality to BOI or inequality~\eqref{eq:renyi_bd} in Theorem~\ref{thm:renyi_main} to obtain an explicit risk bound for the Bayes estimator, we present the following corollary. Recall the definition of the KL neighorhood $B_n(\theta_0,\varepsilon;\theta_0)$ defined in \eqref{eq:chEN}.

\begin{corollary}\label{cor:renyi}
Suppose $\varepsilon \in (0, 1)$ satisfies $n \, \varepsilon^2 > 2$ and $D > 1$. With $\bbP^{(n)}_{\theta_0}$ probability at least $1 - 2/\{(D-1)^2 n \varepsilon^2\}$, 
\begin{align}\label{eq:cor_renyi}
 \int \Big\{\frac{1}{n} D^{(n)}_{\alpha}(\theta, \theta_0) \Big\}\, \Pi_{n, \alpha}(d\theta\,|\,X^{(n)}) \le \frac{D\, \alpha}{1-\alpha} \, \varepsilon^2  + \Big\{ - \frac{1}{n(1-\alpha)} \log \Pi_n(B_n(\theta_0,\varepsilon;\theta_0)) \Big\}.
\end{align}
In particular, if we let $\varepsilon_n$ to be the Bayesian critical radius that is a stationary point of
\begin{align*}
 - \frac{\log \Pi_n(B_n(\theta_0,\varepsilon;\theta_0)) }{n\, \varepsilon} = \varepsilon,
\end{align*}
then with $\bbP^{(n)}_{\theta_0}$ probability at least $1 - 2/\{(D-1)^2 n \varepsilon_n^2\}$, 
\begin{align*}
 \int \Big\{\frac{1}{n} D^{(n)}_{\alpha}(\theta, \theta_0) \Big\}\, \Pi_{n, \alpha}(d\theta\,|\,X^{(n)}) \le \frac{D\,  \alpha + 1}{1-\alpha} \, \varepsilon_n^2.
\end{align*}
\end{corollary}
\noindent The main idea of the proof is to choose the probability measure $\rho$ as $\Pi_n(\cdot \cap B_n(\theta_0,\varepsilon;\theta_0))/\Pi_n(B_n(\theta_0,\varepsilon;\theta_0))$; the restriction of the prior $\Pi_n$ to $B_n(\theta_0,\varepsilon;\theta_0)$. Under this choice, we have $D(\rho\, \vert \vert\,  \Pi_n) = - \log \Pi_n(B_n(\theta_0,\varepsilon;\theta_0))$, and $\int r_{n}(\theta,\,\theta_0) \,\rho(d\theta)$ can be bounded by applying Chebyshev's inequality. If higher moment constraints on the likelihood ratio $r_{n}(\theta,\,\theta_0)$ is also included into the definition of $B_n(\theta_0,\varepsilon;\theta_0)$ in \eqref{eq:chEN}, then the probability bound for \eqref{eq:cor_renyi} to hold can be boosted (for details, see Section 2 in \cite{GhosalVandervaart:07}).

According to Corollary~\ref{cor:renyi}, the overall risk bound in~\eqref{eq:cor_renyi} is a balance between two terms: an approximation error term $\varepsilon_n^2$ and a local complexity measure term $-\frac{1}{n}\log \Pi_n(B_n(\theta_0,\varepsilon;\theta_0))$. For this reason, we will refer to the second term as the \emph{local Bayesian complexity}. The  local Bayesian complexity reflects the compatibility between the prior distribution and the parameter space: if $\Pi_n$ is close to a uniform distribution over $\Theta$, then $-\log \Pi_n(B_n(\theta_0,\varepsilon;\theta_0)) = \log\big\{1 / \Pi_n(B_n(\theta_0,\varepsilon;\theta_0)\}$ is roughly the logarithm of the number of $\varepsilon_n$-balls needed to cover a neighborhood of $\theta_0$, and therefore is related to the local covering entropy. On the other hand, if some prior knowledge about $\theta_0$ is available, then we can combine these knowledge to increase the prior mass around $\theta_0$, which may significantly boost the rate of convergence of the Bayes estimator. This observation is consistent with our previous intuition that averaging based (average case analysis) Bayesian approaches sometimes can be better than optimization based (worst case analysis) frequentist approaches. For example, when certain hyperparameter or tuning parameter, such as the regularity of a function class or sparsity level of a regression model, is unknown, then a Bayesian procedure naturally achieves adaptation to those unknown parameters by placing a prior on them that distributes proper weights to different levels of the hyperparameter (see our examples in Section \ref{sec:exp}). In contrast, a common way to select a tuning parameter in frequentist methods is via cross-validation or data-splitting. These approaches only uses some proportion of data to do estimation, after learning the tuning parameter via the rest, which may not be the most efficient way to use data.

\

Although Theorem~\ref{thm:renyi_main} is useful for obtaining an BOI, when transformed into form \eqref{eq:FOI} the resulting leading constant $c$ of the approximation error term in the BOI is typically strictly larger than $1$, resulting in a non-sharp oracle inequality. Here, we call an oracle inequality sharp if the leading constant $c$ in  \eqref{eq:FOI} is $1$; see, for example, \cite{dalalyan2012}. To solve this issue for the PAC-Bayes inequality in Theorem~\ref{thm:renyi_main}, we consider a second class of PAC-Bayes inequalities that directly characterizes the closeness between $\theta$ and the best approximation $\theta^\ast$ of $\theta_0$ from $\Theta$.

\begin{theorem}[PAC-Bayes inequality relative to $\theta^\ast$]\label{thm:renyi_mainb}
Fix $\alpha \in (0, 1)$. Then, for any $\varepsilon \in (0, 1)$, %it holds that
\begin{equation}\label{eq:renyi_bd_b}
\begin{aligned}
 \int \Big\{\frac{1}{n} D^{(n)}_{\theta_0,\alpha}(\theta, \theta^\ast) \Big\}\, \Pi_{n, \alpha}(d\theta\,|\,X^{(n)}) \le \frac{\alpha}{n (1 - \alpha)} \int r_{n}(\theta,\,\theta_\ast) \,\rho(d\theta) + &\frac{1}{n(1-\alpha)} D(\rho \,,\,\Pi_n)  + \frac{1}{n(1-\alpha)}  \log(1/\varepsilon),\\
  &\mbox{$\forall$ probability measure $\rho\ll \Pi_n$},
\end{aligned}
\end{equation}
with $\bbP^{(n)}_{\theta_0}$ probability at least $(1 - \varepsilon)$. 
\end{theorem}

Similar to Corollary~\ref{cor:renyi} for a concrete Bayesian risk bound for characterizing the closeness between $\theta$ and $\theta_0$, we have the following counterpart for $\theta$ and $\theta^\ast$.

\begin{corollary}\label{cor:renyi_b}
For any $\varepsilon \in (0, 1)$ satisfying $n \, \varepsilon^2 > 2$ and $D > 1$, with $\bbP^{(n)}_{\theta_0}$ probability at least $1 - 2/\{(D-1)^2 n \varepsilon^2\}$, 
\begin{align}\label{eq:cor_renyi_b}
 \int \Big\{\frac{1}{n} D^{(n)}_{\theta_0,\alpha}(\theta, \theta^\ast) \Big\}\, \Pi_{n, \alpha}(d\theta\,|\,X^{(n)}) \le \frac{D\, \alpha}{1-\alpha} \, \varepsilon^2  + \Big\{ - \frac{1}{n(1-\alpha)} \log \Pi_n(B_n(\theta^\ast,\varepsilon;\theta_0)) \Big\}.
\end{align}
In particular, if we let $\varepsilon_n$ to be the Bayesian critical radius that is a stationary point of
\begin{align*}
 - \frac{\log \Pi_n(B_n(\theta^\ast,\varepsilon;\theta_0)) }{n\, \varepsilon} = \varepsilon,
\end{align*}
then with $\bbP^{(n)}_{\theta_0}$ probability at least $1 - 2/\{(D-1)^2 n \varepsilon_n^2\}$, 
\begin{align*}
 \int \Big\{\frac{1}{n} D^{(n)}_{\theta_0,\alpha}(\theta, \theta^\ast) \Big\}\, \Pi_{n, \alpha}(d\theta\,|\,X^{(n)}) \le \frac{D\,  \alpha + 1}{1-\alpha} \, \varepsilon_n^2.
\end{align*}
\end{corollary}

\noindent We now illustrate how Corollary~\ref{cor:renyi_b} leads to a sharp oracle inequality in the misspecified case (a concrete example is provided in Section~\ref{Section:NPregression}).   As noted previously, an oracle inequality is sharp in the misspecified case if the leading constant in front of the model space approximation term is $1$, i.e., 
$d(\hat{\theta}, \theta_0) \leq \inf_{\theta \in \Theta} d(\theta,  \theta_0)  + C \varepsilon_n^2$  for some 
distance metric  $d(\cdot, \cdot)$.   
%This approach to measure the performance of estimators under model misspecification was pioneered by \cite{vapnik1974theory} under the name  of {\em criterion on the minimax  loss}.  
In statistical learning theory, the {\em regret} \cite{vapnik1974theory,rakhlin2013empirical} of an estimator is defined as $d(\hat{\theta}, \theta_0)\}  -  \inf_{\theta \in \Theta} d(\theta,  \theta_0)$.  A benchmark to compare regrets for different estimators is the {\em minimax regret} defined as $\min_{\hat{\theta}} \max _{\theta_0} \big[ \bbE_{\theta_0} \{d(\hat{\theta}, \theta_0)\}  -  \inf_{\theta \in \Theta} d(\theta,  \theta_0)\big]$.  Regret bounds (misspecified case) are substantially harder to obtain compared to minimax risk bounds (well-specified case), and the rate of minimax regret can be different from the minimax risk \cite{rakhlin2013empirical}. Our general technique to derive a sharp oracle inequality for the Bayes estimator will imply that the Bayes estimator has minimax regret.  

Suppose we are interested in certain metric $d_n$, the  square of which is weaker than the average $\alpha$-divergence $\frac{1}{n} D^{(n)}_{\theta_0,\alpha}(\theta, \theta^\ast)$, that is
\begin{align*}
\frac{1}{n} D^{(n)}_{\theta_0,\alpha}(\theta, \theta^\ast) \geq c_\alpha\, d^2_n(\theta,\theta^\ast), \quad \theta\in \Theta,
\end{align*}
where $c_\alpha$ is some positive constant that may depend on $\alpha$. For simplicity, we assume that $\theta^\ast$ is also the minimizer of $d_n(\theta,\theta_0)$ over $\theta\in\Theta$. If this is not the case, then we can always add an extra remainder term to the upper bound that characterizes the difference between $\theta^\ast$ and the best approximation of $\theta_0$ from $\Theta$ relative to $d_n$. Under these assumptions, Corollary~\ref{cor:renyi_b} implies that with high probability, the Bayes estimator $\widehat{\theta}_B$ satisfies
\begin{align*}
d_n(\widehat{\theta}_B, \,\theta^\ast) \leq c_{\alpha}' \, \varepsilon_n,
\end{align*}
where $\varepsilon_n$ is the Bayesian critical radius. Now adding $d_n(\theta^\ast,\,\theta_0)$ to both sides of this inequality and applying the triangle inequality, we obtain
\begin{align*}
d_n(\widehat{\theta}_B, \,\theta_0) \leq \inf_{\theta\in\Theta} d_n(\theta, \,\theta_0) + c_{\alpha}' \, \varepsilon_n,
\end{align*}
which is a sharp oracle inequality. Sometimes, we may be interested in obtaining an oracle inequality for the squared loss $d_n^2$, when $\Theta$ is a vector space and $d_n$ is induced by an inner product, denoted by $\langle \cdot,\cdot\rangle_n$. This is a more intricate problem, as the trivial bound $d_n^2(\widehat{\theta}_B, \theta_0) \le 2[ d_n^2(\widehat{\theta}_B, \theta^\ast) + d_n^2(\theta^\ast, \theta_0)]$ renders the oracle inequality non-sharp. However, 
it is usually true when $\Theta$ is a convex set that
\begin{align*}
\frac{1}{n} D^{(n)}_{\theta_0,\alpha}(\theta, \theta^\ast) \geq c_\alpha\, \big( d^2_n(\theta,\theta^\ast) +2\, \langle  \theta-\theta^\ast, \theta^\ast - \theta_0 \rangle_n\big), \quad\forall ~ \theta\in \Theta.
\end{align*}
For example, this inequality holds for regression with fixed design, where $d_n$ is the $L_2$ empirical norm (details can be found in Section~\ref{Section:NPregression}).
Again, by applying Corollary~\ref{cor:renyi_b} and adding  $d_n^2(\theta^\ast,\,\theta_0)$ to both sides of this inequality, we obtain
\begin{align*}
d^2_n(\widehat{\theta}_B, \,\theta_0) = d^2_n(\widehat{\theta}_B, \,\theta^\ast) + 2\, \langle  \widehat{\theta}_B -\theta^\ast, \theta^\ast - \theta_0 \rangle_n +  d^2_n(\theta^\ast, \,\theta_0)  \leq \inf_{\theta\in\Theta} d^2_n(\theta, \,\theta_0) + c_{\alpha}' \, \varepsilon_n^2,
\end{align*}
which is a sharp oracle inequality for  the squared loss $d_n^2$. 
As an application of this technique,  we derive 
a sharp oracle inequality for estimating a convex function in Theorem \ref{thm:conv:rate_mis}
 when $f_0$ is not necessarily convex.  
 
\paragraph{Comparisons with previous work:}
The most relevant PAC-Bayes type result to ours, such as Theorem~\ref{thm:renyi_main}, is the Theorem 1 in \cite{Dalalyan2008}, which focus on the regression setting $Y_i=f(x_i)+w_i$, where $\theta = f$ is the unknown regression function to be estimated, $x_i$'s are the fixed design points and $w_i$'s are the i.i.d.~zero mean noise, corresponding to the i.n.i.d.~observations. They propose to use the posterior mean of the following quasi-likelihood function as the estimator,
\begin{align*}
L_{n,\beta}(f) = \exp\Big\{-\frac{1}{2\beta} \sum_{i=1}^n (Y_i-f(x_i))^2\Big\},
\end{align*}
where according to their terms, $\beta$ is a temperature parameter. In the special case when $w_i\sim N(0,\sigma^2)$ and $\beta=\sigma^2$, this function reduces to the likelihood function. They establish a PAC-Bayes inequality 
\begin{align*}
\bbE^{(n)}_{\theta_0}[\|\hat{f} - f_0\|_n^2] \leq \int \|f - f_0\|_n^2 \rho(df) + \frac{\beta}{n} D(\rho\,,\,\Pi_n),\ \ \ \mbox{$\forall$ probability measure $\rho\ll \Pi_n$},
\end{align*}
when $\beta\geq 4\sigma^2$, where $\hat{f}$ is the corresponding posterior mean. Therefore, their quasi-likelihood approach can be viewed as a special case under our fractional posterior with $\alpha \leq 1/4$.
Their proof is specialized to the empirical $L^2(\bbP_n)$ loss and requires the log-likelihood function to also take a sum of squares form. In contrast, our PAC-Bayes inequality generalizes the results in \cite{Dalalyan2008} to a more broader class of models. Moreover, the posterior expectation in $\int R(f,f_0) \Pi_{n,\alpha}(df\,|\,X^{(n)})$ in our PAC-Bayes inequality is taken outside the loss function $R$ instead of plugging in the estimator as $R(\widehat{f}_B, f_0)$, which is always bounded above by $\int R(f,f_0) \Pi_{n,\alpha}(df\,|\,X^{(n)})$ when $R(f, f_0)$ is a convex function of $f$.

\section{Examples}\label{sec:exp}

In this section, we demonstrate the salient features of our theory through a number of illustrative examples. All results in this section are stated in the form of PAC-Bayes bounds derived from Corollary \ref{cor:renyi}. We note that one could alternatively use Theorem \ref{Thm:contraction} to obtain similar conclusions. Our first example illustrates the efficacy of the fractional posterior approach in shape-constrained estimation. In the well-specified case, we demonstrate the optimal concentration in estimating a convex function, where we bypass the need to compute the covering entropy of the convex function space. To best of our knowledge, such a result is not available in the Bayesian literature. In the model misspecified case, we derive a sharp Bayesian oracle inequality that extends the recent sharp oracle inequality for one-dimensional convex regression obtained in \cite{Bellec:2015} to general dimension $d\geq 1$.
The next two examples concern the classical nonparametric regression and nonparametric density estimation, where we show that the fractional posterior optimally and adaptively concentrates at the true parameter value with substantially relaxed assumptions on the prior compared to existing theory. These two examples also theoretically justify the adaptation nature of integration based Bayesian approaches.

%In the following, we consider a few applications corresponding to the specified case to demonstrate the usefulness of our theory.  The examples show that one can obtain optimal posterior concentration using the power likelihood with substantially relaxed assumptions on the prior distribution.  This is primarily because the set of sufficient conditions pertain to only the prior concentration (refer for example to  Theorem \ref{Thm:contraction}  for the concentration of the posterior measure and Corollary \ref{cor:renyi} for the PAC-Bayes risk bounds in the specified case).  For simplicity of exposition,  we only provide the PAC-Bayes risk bounds in the examples.  

\subsection{Nonparametric regression}\label{Section:NPregression}

We start with a general regression problem.
Consider the following nonparametric regression model with fixed design
\begin{align}\label{eq:sgp_def}
y_i = \mu(x_i) + \varepsilon_i, \quad \varepsilon_i \sim \mx N(0,\sigma^2), \quad i = 1, \ldots, n,
\end{align}
where $y_i \in \mathbb{R}$ is the response, $x_i \in [0, 1]^d$ is the $i$th fixed design point, $\mu : [0, 1]^d \to \mb R$ is the unknown regression function to be estimated and $\sigma$ is the noise level. For simplicity, we assume that $\sigma$ is known.  Given our general notation, $\mu$ plays the role of $\theta$ here. To estimate $\mu$, we place a prior $\Pi$ over some function space $\m F$. An examination of the fractional likelihood $L_{n, \alpha}(\mu)$ and the corresponding posterior $\Pi_{n, \alpha}(\mu)$ under \eqref{eq:sgp_def} yields    
\begin{align*}
\Pi_{n, \alpha}(\mu) =  \frac{ \prod_{i=1}^n \phi_{\sigma}^{\alpha}\{y_i - \mu(x_i)\} \Pi(\mu)}
{\int \prod_{i=1}^n \phi_{\sigma}^{\alpha}\{y_i - \mu(x_i)\} \Pi(\mu)}  =  
 \frac{ \prod_{i=1}^n \phi_{\psi}\{y_i - \mu(x_i)\} \Pi(\mu)}
{\int \prod_{i=1}^n \phi_{\psi}\{y_i - \mu(x_i)\} \Pi(\mu)} 
\end{align*}
where $\psi = \sigma/\sqrt{\alpha}$ and $\phi_{\sigma}$ stands for the multivariate normal density with zero mean and covariance matrix $\sigma^2 I_n$.  Hence the fractional posterior for \eqref{eq:sgp_def} is essentially a standard posterior with a different variance parameter in the likelihood.  

We use the notation $\|\cdot\|_{2,n}$ to denote the $L_2(\mb P_n)$ norm relative to the empirical measure $\mb P_n=n^{-1}\,\sum_{i=1}^n \delta_{x_i}$, and use $\langle\cdot,\cdot\rangle_n$ to denote the empirical inner product, that is, $\langle f,g\rangle_n=n^{-1}\sum_{i=1}^n f(x_i)\, g(x_i)$ for two functions $f$ and $g$.
Let $\mu_0$ denote the true regression function, which may be in or not in $\m F$, depending on whether the model is well-specified or misspecified. The observations $\{y_i\}_{i=1}^n$ are i.n.i.d., and simple calculations yields
\begin{align*}
D(p^{(n)}_{\mu_0},p^{(n)}_{\mu}) = \frac{n}{2} \,\|\mu - \mu_0\|_{2,n}^2,
\end{align*}
therefore, the function $\mu^\ast \in \m F$ that minimizes the KL divergence is
\begin{align*}
\mu^\ast:\,= \arg\inf_{\mu \in\m F} \|\mu-\mu_0\|_{2,n}^2.
\end{align*}
Moreover, straightforward calculations show 
\begin{equation}\label{Eqn:formula_convex}
\begin{aligned}
&B_n(\mu^\ast, \varepsilon; \mu_0) \supset \Big\{ \mu\in\m F:\,  \|\mu-\mu^\ast\|_{2,n}^2 + 2 \langle \mu- \mu^\ast,\, \mu^\ast -\mu_0\rangle_n \leq \varepsilon^2\Big\},\\
& D_{\mu_0,\alpha}^{(n)}(\mu,\, \mu^\ast) = \frac{n \alpha}{2(1-\alpha)} \bigg[ \|\mu-\mu_0\|_{2,n}^2 -  \|\mu^\ast-\mu_0\|_{2,n}^2 - \alpha\,  \|\mu-\mu^\ast\|_{2,n}^2\bigg].
\end{aligned}
\end{equation}

\

Let $K := \{ \widetilde{\mu} = (\mu(x_1), \ldots, \mu(x_n))^{\T} : \mu \in \m F\} \subset \mb R^n$ denote the parameter space for the $n$-vector $\widetilde{\mu}$ created by evaluating the function $\mu$ at the design points. $D_{\mu_0,\alpha}^{(n)}(\mu,\, \mu^\ast)$ in \eqref{Eqn:formula_convex} can be equivalently expressed as $T \alpha/\{2 (1 - \alpha)\}$, with
$$
T = \bigg[ \| \widetilde{\mu} - \widetilde{\mu}_0\|^2 -  \| \widetilde{\mu}^*- \widetilde{\mu}_0\|^2 - \alpha\,  \| \widetilde{\mu}- \widetilde{\mu}^*\|^2\bigg],
$$
where $\widetilde{\mu}^* = \mbox{Proj}_K(\widetilde{\mu}_0)$, the Euclidean projection of $\tilde{\mu}_0$ to the set $K$. If $K$ is a closed convex set in $\mb R^n$, then it is a standard fact from convex geometry (see, for example, \cite{rockafellar1997convex}) that the projection is uniquely defined and satisfies 
$$
\dotp{ \widetilde{\mu}_0 - \widetilde{\mu}^*}{ \widetilde{\mu} - \widetilde{\mu}^*} \le 0, \quad \forall \, \mu \in K.
$$
With some simple algebra, $T =  (1 - \alpha) \| \widetilde{\mu}-\widetilde{\mu}^\ast\|^2 + 2\langle \widetilde{\mu} -\widetilde{\mu}^\ast, \widetilde{\mu}^\ast-\widetilde{\mu}_0\rangle \ge 0$ by the above inequality. Hence, $D_{\mu_0,\alpha}^{(n)}(\mu,\, \mu^\ast)$ defines a valid divergence measure if $K$ is convex. It is straightforward to verify that $K$ is convex if the class of functions $\m F$ is monotone or convex. %Thus, we can use our general divergence measure to derive sharp oracle inequalities in 

%and define $\widetilde{\mu^*} = \mbox{Proj}_K(\widetilde{\mu}_0)$. 

\paragraph{Shape-constrained function estimation:} The observations in the previous paragraph suggest the applicability of our framework to shape-constrained regression problems. We provide an illustration via convex regression, where $\m F$ is the function space of all $d$-dimensional convex functions over $[0,1]^d$.
Our fractional posterior framework becomes especially attractive in such problems, since it obviates the need to compute entropy numbers in restricted spaces, which can be a challenging exercise in itself \cite{guntuboyina2013covering}.
%In this section,  we show a power-likelihood  delivers optimal Bayes risk bounds in shape-constrained inference,  without having to compute entropy of compact subsets of restricted spaces.  We consider the same regression model \eqref{eq:sgp_def}, but $\mu:[0, 1]^d \to \mathbb{R}$ is now assumed to be convex.  

It is recent practice in the frequentist literature to avoid additional smoothness assumptions on convex functions while studying  rates of convergence \cite{guntuboyina2015global,balazs2015near}. 
To that end, let $\partial \mu(x)$ denote the {\em sub-gradient} of the function $\mu$ at the point $x$, that is, 
\begin{eqnarray*}
\partial \mu(x)  = \{ s \in \mathbb{R}^d:  \mu(z) \geq \mu(x) + s^{\T}(z-x),\, \text{for all}\, z \in [0, 1]^d\}.  
\end{eqnarray*}
As in \cite{balazs2015near}, define the class of convex, sub-differentiable, uniformly Lipschitz functions on $[0, 1]^d$ as 
\begin{eqnarray*}
\mbox{Co}_L[0, 1]^d  = \{\mu: [0, 1]^d  \to \mathbb{R}, \, \mu\, \text{is convex}, \, \norm{s} \leq L \, \text{for all}\, s \in \partial \mu(x), \,   \partial \mu(x) \, \text{is non empty for all}\,  x\}.   
\end{eqnarray*}
We model $\mu$ as a maximum of hyperplanes \cite{hannah2011bayesian}, with a prior distribution for the number of affine functions over which the maximum is taken. Specifically, we let
\begin{eqnarray}\label{eq:priorconv}
\mu(x) \mid k, \{a_j^k, b_j^k\}  = \max_{j \in \{1, \ldots, k\}} \big[ (a_j^k)^{\T} x  +  b_j^k\big], \quad \{(a_j^k)^{\T},  b_j^k\}^{\T}  \mid k\sim \mbox{N}(0, \tau^2 \mathrm{I}_{d+1}), \quad k \sim \pi_k.  
\end{eqnarray}

%To develop a Bayesian model for estimating convex functions,  we adopt a common technique of 
%maximizing  affine functions with a prior for the number of components over which the maximum is taken.  We propose the following prior distribution for $\mu$
The following theorem shows that in the well-specified case where $\mu_0\in\m F$, with no additional smoothness condition on $\mu_0$,  we obtain a Bayes risk bound  of the order $n^{-2/(4 + d)}$ up to logarithmic terms, which coincides with the minimax risk under any $d\geq1$ \cite{guntuboyina2015global,balazs2015near}.   
\begin{theorem}[Bayesian risk, well-specified case]\label{thm:conv:rate}
Consider the model \eqref{eq:sgp_def} with $\mu_0 \in \mbox{Co}_L[0, 1]^d$, and the prior for $\mu$  satisfies \eqref{eq:priorconv} with $\pi_k \geq 
\exp \{-Ck \log k\}$ for some constant $C > 0$,  then with $\bbP_{\mu_0}^{(n)}$ probability tending to one, 
\begin{align}\label{npreg:rate}
\int \norm{\mu - \mu_0}_{2, n}^2 d\Pi_{n, \alpha}(\mu)  \leq \frac{C}{\alpha\,(1-\alpha)} \varepsilon_n^2, 
\end{align}
where $\varepsilon_n = n^{-2/(4 +d)} \log^{t} n$ with $t = 2/(4+d)$, and $C$ is some constant independent of $\alpha$.
\end{theorem}
%An inspection of the proof of Theorem \ref{thm:conv:rate} shows that we only require the prior concentration $\Pi(\norm{\mu - \mu_0}_{\infty}  \leq \varepsilon_n) \geq \exp \{ - n \varepsilon_n^2\}$ to achieve the optimal risk bound.  The proof bypasses the entropy calculations for compact subsets of  $\mbox{Co}[0, 1]^d$ which is known to be a challenging exercise. In fact,  there are existing full-length articles \cite{guntuboyina2013covering} devoted to just  finding entropy numbers of convex functions.   

Now we consider the misspecified case, where $\mu_0$ maybe a non-convex function. In this case, $\mu^\ast$ is the projection of $\mu_0$ into $\mbox{Co}_L[0, 1]^d$ relative to the $\|\cdot\|_{2,n}$ norm. We obtain the following sharp Bayesian oracle inequality, which generalizes the result of one-dimensional convex regression obtained in \cite{Bellec:2015} to general dimension $d\geq 1$.
\begin{theorem}[Bayesian risk, misspecified case]\label{thm:conv:rate_mis}
Consider the model \eqref{eq:sgp_def} with $\mu_0 \in C[0,1]^d$, and the prior for $\mu$ satisfying the conditions in Theorem~\ref{thm:conv:rate}.  Then with $\bbP_{\mu_0}^{(n)}$ probability tending to one, 
\begin{align}\label{npreg:rate}
\int \norm{\mu - \mu_0}_{2, n}^2 d\Pi_{n, \alpha}(\mu)  \leq \inf_{\mu \in \mbox{Co}_L[0, 1]^d}  \norm{\mu - \mu_0}_{2, n}^2 + \frac{C}{\alpha\,(1-\alpha)} \varepsilon_n^2,
\end{align}
where $\varepsilon_n$ is given in Theorem~\ref{thm:conv:rate}, and $C$ is some constant independent of $\alpha$.
\end{theorem}
\noindent This sharp oracle inequality implies some geometric structure of the fractional posterior that cannot be obtained via a non-sharp one. For example, as an immediate consequence, if we let $\mu^\ast$ be any minimizer of $ \norm{\mu - \mu_0}_{2, n}^2 $ over $ \mbox{Co}_L[0, 1]^d$, then \eqref{npreg:rate} implies
\begin{align*}
\int \langle \mu - \mu^\ast, \mu^\ast - \mu_0 \rangle_n  \, d\Pi_{n, \alpha}(\mu)  \leq  \frac{C}{2\, \alpha\,(1-\alpha)} \varepsilon_n^2.
\end{align*}
Since $ \langle \mu - \mu^\ast, \mu^\ast - \mu_0 \rangle_n$ is nonnegative for all $\mu \in \mbox{Co}_L[0, 1]^d$ due to the convexity of  $\mbox{Co}_L[0, 1]^d$ (see, for example, \cite{rockafellar1997convex}), this display suggests that the fractional posterior distribution of $\mu$ tends to concentrate on the narrow cone with vertex $\mu^\ast$ consisting of points such that the angle between vectors $\mu-\mu^\ast$ and $\mu^\ast -\mu_0$ is of order 
\begin{align*}
\frac{\pi}{2} - \frac{\langle \mu - \mu^\ast, \mu^\ast - \mu_0 \rangle_n}{\|\mu - \mu^\ast\|_n\cdot \|\mu^\ast - \mu_0\|_n} \approx \frac{\pi}{2} -\frac{\varepsilon_n}{\|\mu^\ast - \mu_0\|_n}\approx \frac{\pi}{2}.
\end{align*}
That is, with large fractional posterior probability, $\mu-\mu^\ast$ is almost perpendicular to $\mu^\ast -\mu_0$.

We conjecture that the current technique will allow us to find sharp oracle inequalities for Bayes estimators in monotone function estimation as well. Sharp oracle inequalities for isotonic regression for $d = 1$ has been recently established in \cite{Bellec:2015}, improving on a previous risk bound by \cite{chatterjee2015risk}. We leave this as a topic for future research. 
%In the context of monotone function estimation,  \cite{chatterjee2015risk}  obtained a non-sharp oracle inequality (with leading constant $> 1$). The result are improved by \cite{Bellec:2015}, where sharp oracle inequalities are obtained for the monotone least squares estimator.  We conjecture that the current technique will allow us to find sharp oracle inequalities for corresponding Bayes estimators. Such an investigation is deferred as a topic of separate research.  

\paragraph{Nonparametric GP regression:}
In this example, we consider the regression model \eqref{eq:sgp_def} with function space $\m F = C[0,1]^d$.
 We assign to $\mu$ a Gaussian process prior with mean function $h_\mu: [0, 1]^d \to \mathbb{R}$ and covariance kernel 
$c(x, x')$, a positive definite function from $[0, 1]^d \times [0, 1]^d \to \mathbb{R}$.   We denote the prior by $\mu \sim \mbox{GP}(h_\mu, c)$.  Without loss of generality we assume $h_\mu \equiv 0$. We work with the squared exponential covariance kernel  $c_a(x, x') = e^{-a^2 \norm{x - x'}^2}$, with the prior for $a$ 
satisfying 
\begin{align}\label{eq:ga}
g(a) \geq A_1 a^p e^{-B_1a^d \log^q a}. 
\end{align}
With this assumption, we show that the fractional posterior concentrates at the minimax rate (up to logarithmic terms) adaptively over $\mu_0 \in C^{\beta}[0, 1]^d$, where $\beta$ is the unknown smoothness level of $\mu_0$. To obtain the same result for the usual posterior, \cite{van2009adaptive} require the prior on $a$ to additionally satisfy an upper bound of the same order as the lower bound in \eqref{eq:ga}, once again, ruling out heavy tailed priors. 

\begin{theorem}\label{thm:npreg:rate}
Consider the model \eqref{eq:sgp_def}, with a conditional GP prior $\mu \mid a \sim \mbox{GP}(0, c_a)$ and suppose $a \sim g(\cdot)$ satisfies \eqref{eq:ga}. 
If the true function $\mu_0 \in C^{\beta}[0, 1]^d$, then \eqref{npreg:rate} is satisfied with 
$\varepsilon_n = n^{-\beta/(2\beta +d)} (\log n)^t$, where $t=  \{(1+d) \vee q\}/ (2 + d/\alpha)$. 
\end{theorem}
\noindent Theorem \ref{thm:npreg:rate} can be extended to other kernels in a straightforward manner.

\subsection{Nonparametric density estimation}

We make the same local H\"older smoothness assumptions on the true density as in \cite{ShenTokdarGhosal2013}. For $\mathcal{Y} \subset \mathbb{R}^d$, a function $L:\mathcal{Y}\rightarrow[0,\infty)$, and $\tau_0, \beta \ge 0$, 
the class of locally H\"older functions with smoothness $\beta$, denoted $\mathcal{C}^{\beta,L,\tau_0}$, consists of 
$f:\mathbb{R}^d\rightarrow \mathbb{R}$ such that for any $k=(k_1,\ldots,k_d)$ with $k_1+\cdots+k_d \leq \lfloor \beta \rfloor$, the mixed partial derivative $D^k f$ of order $k$ is finite, and for $k_1+\cdots+k_d = \lfloor \beta \rfloor$ and $\Delta y \in \mathcal{Y}$,
\[
|D^k f (y+\Delta y) - D^k f (y)| \leq L(y) \norm{\Delta y}^{\beta-\lfloor \beta \rfloor} e^{\tau_0 ||\Delta y||^2}.
\]

% We formulate two sets of assumptions, one on the data generating process and the other on the prior. \\
\noindent {\bf Assumptions on the true density:}
\label{sec:asns}
We assume $f_0 \in \mathcal{C}^{\beta,L,\tau_0}$ and  for all $k \leq \lfloor \beta \rfloor$ and some $\varepsilon>0$,
	\begin{equation}
	\label{eq:asnE0Dff0_Lf0}
	\int_{\mathcal{Y}} \left|\frac{D^k f_0(y)}{f_0(y)}\right|^{(2\beta+\varepsilon)/k} f_0(y) dy< \infty,
	\;
	\int_{\mathcal{Y}} \left|\frac{L(y)}{f_0(y)}\right|^{(2\beta+\varepsilon)/\beta} f_0(y) dy< \infty.
	\end{equation} 
Moreover, for all sufficiently large $y \in \mathcal{Y}$ and some positive $(c,b,\tau)$, 
\begin{equation}
	\label{eq:asnf0_exp_tails}
f_0(y) \leq c \exp(-b \norm{y}^\tau). 
	\end{equation}  
We model the density $f$ of i.i.d. observations $X_1, \ldots, X_n$ via a mixture of finite mixtures (MFM; \cite{miller2015mixture}), which is a finite mixture model with a prior on the number of mixture components. With some minor modifications, the results can be adapted to infinite mixture models, such as Dirichlet process mixtures \cite{KruijerRousseauVaart:09,ShenTokdarGhosal2013}. Our concentration results can accommodate {\em heavy tailed} prior distributions on the component specific means. 

%In the following, we use Corollary \ref{cor:renyi} to obtain optimal risk bounds for the Bayes estimator under a wider variety of prior distributions compared to existing literature \cite{ShenTokdarGhosal2013}.  
%In particular, our theory can accommodate heavier tailed prior (e.g. densities with polynomial tails) as opposed to the existing theory of posterior contraction; refer for example to  \cite{KruijerRousseauVaart:09,ShenTokdarGhosal2013}.   

\noindent {\bf Prior:} \label{sec:prior} 
We model the unknown density by a location mixture of normal densities 
\begin{equation}
\label{eq:cond_mix_def}
p(y \mid \psi, m) = \sum_{j=1}^m \alpha_j  \phi_{\mu_j,\sigma}(y),
\end{equation}
and a prior $\Pi$ is induced on the space of densities by assigning a 
a prior on $m \in \mathbb{N}$ and given $m$, a prior on $\psi =\{ (\mu_j, \alpha_j)_{j=1}^m, \sigma \}$,
where $\sigma \in (0,\infty)$, 
$\mu_j \in \mathbb{R}^{d}$ for $j = 1, \ldots, m$, and $(\alpha_1, \ldots, \alpha_m) \in \Delta^{m-1}$. 
In the sequel, $a_i$s denote positive constants. We assume the prior for $\sigma$ satisfies 
\begin{eqnarray}  
	\Pi\{ s < \sigma^{-2} < s(1+t) \} &\ge& a_6 s^{a_7} t^{a_8} \exp \{-a_9 s^{1/2}\}, \quad s > 0, \quad t \in (0,1). 
	\label{eq:asnPrior_sigma3}
\end{eqnarray}
We assign a  Dirichlet$(a/m,\ldots,a/m)$ prior for $(\alpha_1,\ldots,\alpha_m)$ given $m$, where $a > 0$ is a fixed constant, and let
\begin{equation}
	\label{eq:asnPrior_m}
	\Pi(m=i) \propto \exp(-a_{10} i (\log i)^{\tau_1}), \quad i =2, 3, \ldots, \quad \tau_1 \geq 0.
\end{equation}
Finally, we assume that the $\mu_{j}$s are independent from other parameters and across $j$, with a prior density satisfying
\begin{equation}
	\label{eq:asnPrior_mu_lb}
	\Pi(\mu_j) \ge a_{11}\exp(-a_{12} \|\mu_j\|^{\tau_2} ), \quad \tau_2 > 0. 
\end{equation}
Let $p_{\theta_0}^{(n)} =  \prod_{i=1}^n f_0(X_i)$ and $p_{\theta}^{(n)} =  \prod_{i=1}^n p(y_i \mid \psi, m)$, where $\theta = \{(\mu_j, \alpha_j)_{j=1}^{\infty}, \sigma, m \}$.   Then  
$\frac{1}{n} D^{(n)}_{\alpha}(\theta, \theta_0)  =   D_{\alpha}(f_0, p (\cdot \mid \psi, m))$, abbreviated by $D_{\alpha}(f_0, p)$.  
\begin{theorem} 
\label{th:rate_cond_dens} 
Assume $f_0$ satisfies \eqref{eq:asnE0Dff0_Lf0}--\eqref{eq:asnf0_exp_tails} and the prior distribution on $\theta$ satisfies
\eqref{eq:asnPrior_sigma3}--\eqref{eq:asnPrior_mu_lb}. Then, with $\bbP^{(n)}_{\theta_0}$ probability at least $1 - 2/\{(D-1)^2 n \varepsilon_n^2\}$, 
\begin{align*}
 \int \big\{D_{\alpha}(f_0, p)\big\} \, \Pi_{n, \alpha}(d\theta\,|\,X^{(n)}) \le \frac{D\,  \alpha + 1}{1-\alpha} \, \varepsilon_n^2,
\end{align*}
where $\varepsilon_n=Cn^{-\beta/(2\beta+d)} (\log n)^t$ with $t > t_0 +  \max \{0, (1- \tau_1)/2\}$, 
$t_0=  (ds + \max\{\tau_1,1,\tau_2/\tau\}) / (2 + d/\beta)$, 
and  $s = 1 + 1/\beta + 1/\tau$ for some constant $C > 0$. 
\end{theorem}

%Some examples of prior distributions that satisfy \eqref{eq:asnPrior_sigma1}-\eqref{eq:asnPrior_sigma2} are inverse Gamma, exponential  and Cauchy.  
Inverse-gamma, exponential and half-Cauchy families of densities satisfy \eqref{eq:asnPrior_sigma3}. In addition to \eqref{eq:asnPrior_sigma3}, \cite{KruijerRousseauVaart:09,ShenTokdarGhosal2013} also require two-sided prior tail bounds 
\begin{eqnarray}
	\Pi( \sigma^{-2} \geq s)  &\leq& a_1 \exp \{-a_2 s^{a_3} \} \quad \text{for all sufficiently large} \, \,  s > 0, 
	\label{eq:asnPrior_sigma1} \label{eq:Kruijer} \\
	\Pi( \sigma^{-2} < s)  &\leq& a_4 s ^{a_5}  \quad \text{for all sufficiently small}  \, \,  s > 0, \label{eq:asnPrior_sigma2}
\end{eqnarray}
which impose additional restrictions for the prior choice on $\sigma$, in particular, ruling out heavy-tailed priors.  Assumption \eqref{eq:asnPrior_mu_lb} allows heavy-tailed prior distributions on the component specific means $\mu_j$s, which is not permissible in the existing theory due to the additional requirement \cite{KruijerRousseauVaart:09,ShenTokdarGhosal2013} on the tail behavior	
\begin{equation*}
	\label{eq:asnPrior_mu_tail_ub}
	1- \Pi(\mu_{j} \in [-r,r]^{d}) \leq \exp(-a_{13} r^{\tau_3}),
\end{equation*}
for some $a_{13}, \tau_3>0$ and all sufficiently large $r > 0$. 

%The usual conditionally conjugate inverse Gamma
%prior for $\sigma^2$ does not satisfy
% \eqref{eq:asnPrior_sigma3}.  \eqref{eq:asnPrior_sigma3} requires the probability to values of $\sigma$ near $0$ to be 
% higher than the corresponding probability for inverse Gamma prior for $\sigma^2$.  
%This assumption is in line with the previous work on adaptive posterior contraction rates for mixture models, see 
%\cite{KruijerRousseauVaart:09}.   However, in addition to \eqref{eq:asnPrior_sigma3},  \cite{KruijerRousseauVaart:09,ShenTokdarGhosal2013} also require the following upper and lower tail behavior for the prior on $\sigma^{-2}$
%\begin{eqnarray}
%	\Pi( \sigma^{-2} \geq s)  &\leq& a_1 \exp \{-a_2 s^{a_3} \} \quad \text{for all sufficiently large} \, \,  s > 0 
%	\label{eq:asnPrior_sigma1} \label{eq:Kruijer} \\
%	\Pi( \sigma^{-2} < s)  &\leq& a_4 s ^{a_5}  \quad \text{for all sufficiently small}  \, \,  s > 0 \label{eq:asnPrior_sigma2}
%	\end{eqnarray}
%We {\em do not} require \eqref{eq:Kruijer} - \eqref{eq:asnPrior_sigma2} here.  

%In addition to 	\eqref{eq:asnPrior_mu_lb},  \cite{KruijerRousseauVaart:09,ShenTokdarGhosal2013} require a further assumption \eqref{eq:asnPrior_mu_tail_ub} on the tail decay of the prior for $\mu_j$'s given by 	
%	\begin{equation}
%	\label{eq:asnPrior_mu_tail_ub}
%	1- \Pi(\mu_{j} \in [-r,r]^{d}) \leq \exp(-a_{13} r^{\tau_3}). 
%	\end{equation}
%	 for some $a_{13}, \tau_3>0$ and all sufficiently large $r > 0$,
%We {\em do not} require \eqref{eq:asnPrior_mu_tail_ub} to obtain our main theorem on PAC-Bayes risk bound.   

\section{Proofs}\label{section:proof}
In this section, we present proofs of our results in the main sections.
\subsection{Proof of Lemma~\ref{Lem:divergence}}
 It is immediate from the definition of $A^{(n)}_{\theta_0,\alpha}$ that $A^{(n)}_{\theta_0,\alpha}(\theta,\theta^\ast)>0$ for all $\theta\in\Theta$. We prove the second part that $A^{(n)}_{\theta_0,\alpha}(\theta,\theta^\ast)\leq 1$. Let
%\begin{eqnarray*}
%D^\ast_{\theta_0}(\theta^\ast \vert \vert \theta) = D(\bbP^{(n)}_{\theta_0} \vert \vert \bbP^{(n)}_\theta)  -  D(\bbP^{(n)}_{\theta_0} \vert \vert \bbP^{(n)}_{\theta^*}). 
%\end{eqnarray*}
\begin{eqnarray*}
D^\ast(\theta) = D(p^{(n)}_{\theta_0}, p^{(n)}_\theta)  -  D(p^{(n)}_{\theta_0}, p^{(n)}_{\theta^*}), \quad \theta \in \Theta.
\end{eqnarray*}
 By definition of $\theta^*$, $D^*(\theta) \geq 0$ for all $\theta\in\Theta$. 
 For any $\phi  \in [0, \varepsilon]$ and any $\theta\in\Theta$, define $\theta_\phi \in \Theta$ by $p_{\theta_\phi}^{(n)} =  \phi p_{\theta}^{(n)} + (1-\phi) p_{\theta^*}^{(n)}$, where $\varepsilon>0$ is some small enough constant so that $\theta_\phi \in \Theta$ for all $\phi \in [0, \varepsilon]$; existence of such $\varepsilon > 0$ is guaranteed by the assumed condition. Let the mapping $g: [0, \varepsilon] \to \mathbb{R}^+$ given by 
 \begin{eqnarray*}
 g(\phi) = D^*(\theta_\phi)  =  \int p^{(n)}_{\theta_0} \log \bigg\{ \frac{p^{(n)}_{\theta^\ast}}{p^{(n)}_{\theta_\phi}}\bigg\}  d\mu^{(n)} = 
 - \int p^{(n)}_{\theta_0} \log \bigg\{ 1+ \phi \bigg( \frac{p^{(n)}_{\theta}}{p^{(n)}_{\theta^*}}  -1 \bigg) \bigg\}  d\mu^{(n)}.
 \end{eqnarray*}
It follows that $g'(0+) = 1-  A^{(n)}_{\theta_0,1}(\theta, \theta^\ast)$. Using Jensen's inequality, we have $A^{(n)}_{\theta_0,\alpha}(\theta,\theta^\ast) \leq \big\{A^{(n)}_{\theta_0,1}(\theta,\theta^\ast)\big\}^{\alpha}$. Therefore, it suffices to show that $g'(0+)\geq 0$.

Observe from the definition of $\theta^\ast$ that $g(0) =0$ and $g(\phi) \geq 0$ for $\phi \in[0,\varepsilon]$.
Consider the function $h_x: [0, \varepsilon] \to \mathbb{R}$ for any fixed $x$  given by $\phi \mapsto  \log (1+ \phi x)/ \phi$.
 Then $h_x(\phi) \to x$ as $\phi \to 0$, and we have
 $$0 \leq \{g(\phi) - g(0)\}/\phi = - \int p^{(n)}_{\theta_0} h_{\{p^{(n)}_{\theta}/p^{(n)}_{\theta^\ast} -1\}}(\phi) d\mu^{(n)}.$$
For $x \geq 0$, $h_x$ monotonically increases to $x$ as $\phi \downarrow 0$.  We have the decomposition $$\int p^{(n)}_{\theta_0} h_{\{p^{(n)}_{\theta}/p^{(n)}_{\theta^\ast} -1\}}(\phi) d\mu^{(n)}=\int_{p^{(n)}_{\theta}\leq p^{(n)}_{\theta^\ast}} p^{(n)}_{\theta_0} h_{\{p^{(n)}_{\theta}/p^{(n)}_{\theta^\ast} -1\}}(\phi) d\mu^{(n)}+ \int_{p^{(n)}_{\theta}>p^{(n)}_{\theta^\ast}} p^{(n)}_{\theta_0} h_{\{p^{(n)}_{\theta}/p^{(n)}_{\theta^\ast} -1\}}(\phi) d\mu^{(n)}.$$
Now, we can apply the monotone convergence theorem to the first term and bounded convergence theorem to the second term in this display to obtain that $g'(0+)=\lim_{\phi\to 0+} \{g(\phi) - g(0)\}/\phi$ exists and is greater than or equal to zero.

\subsection{Proof of Theorem~\ref{Thm:contraction}}
Let us write
\begin{align*}
U_n:\,= \Big\{\theta\in\Theta:\, D^{(n)}_{\theta_0,\alpha}(\theta,\,\theta^\ast) \geq \frac{D+3t}{1-\alpha}\, n\,\varepsilon_n^2\Big\}.
\end{align*}
Then, we can express the desired posterior probability as
\begin{align}
\Pi_{n,\alpha} \Big(D^{(n)}_{\theta_0,\alpha}(\theta,\,\theta^\ast) \geq \frac{D+3t}{1-\alpha}\, n\,\varepsilon_n^2\ \Big| \, X^{(n)}\Big)  =  \frac{\int_{U_n} e^{-\alpha\, r_{n}(\theta,\,\theta^\ast)} \,\Pi_n(d\theta)}{\int_{\Theta} e^{-\alpha\, r_{n}(\theta,\,\theta^\ast)}\, \Pi_n(d\theta)}.\label{Eq:post}
\end{align} 

Let us first consider the numerator.
By the definition of the $\alpha$-divergence $D^{(n)}_{\theta_0,\alpha}(\theta,\,\theta^\ast)$, we have 
\begin{align}\label{eq:key_div}
\bbE^{(n)}_{\theta_0} e^{-\alpha\, r_{n}(\theta,\,\theta^\ast)}  = A_{\theta_0,\alpha}^{(n)}(\theta, \theta^\ast) = e^{-(1 - \alpha)\,  D^{(n)}_{\theta_0,\alpha}(\theta,\,\theta^\ast)}.
\end{align}
Now integrating both side with respect to the prior $\Pi$ over $U_n$ and applying Fubini's theorem, we can get
\begin{align*}
\bbE^{(n)}_{\theta_0} \int_{U_n} e^{-\alpha\, r_{n}(\theta,\,\theta^\ast)} \Pi_n(d\theta)  = \int_{U_n}  e^{-(1 - \alpha)\, D^{(n)}_{\theta_0,\alpha}(\theta,\,\theta^\ast)} \Pi_n(d\theta) \leq e^{-(D+3t)\, n \, \varepsilon_n^2},
\end{align*}
where the last step follows from the definition of $U_n$.
An application of the Markov inequality yields the following high probability bound for the numerator on the right hand side of~\eqref{Eq:post},
\begin{align} \label{Eqn:NumBound}
\bbP^{(n)}_{\theta_0}\Big[\int_{U_n} e^{-r_{n, \alpha}(f)} \Pi_n(df)  \geq e^{-(D+2t)\, n \, \varepsilon_n^2} \Big] \leq e^{-t\, n \, \varepsilon_n^2} \leq \frac{1}{(D-1+t)^2 \,n\,\varepsilon_n^2}.
\end{align}

Next, we consider the denominator on the right hand side of~\eqref{Eq:post}. We always have the lower bound
\begin{align*}
\int_{\Theta} e^{-\alpha\, r_{n}(\theta,\,\theta^\ast)} \Pi_n(d\theta) \geq \int_{B_{n}(\theta^\ast,\varepsilon_n; \theta_0)} e^{-\alpha\, r_{n}(\theta,\,\theta^\ast)} \Pi_n(d\theta).
\end{align*}
We invoke the following result (Lemma 8.1, \cite{ghosal2000}), which is a high probability lower bound to $ \int_{B_{n}(\theta^\ast,\varepsilon_n; \theta_0)} e^{-\alpha\, r_{n}(\theta,\,\theta^\ast)} \Pi_n(d\theta)$ (which has been adapted to the $\alpha$-fractional likelihood):
for any $D>1$, we have
\begin{align} \label{Eqn:DenBound}
\bbP^{(n)}_{\theta_0}\Big[\int_{B_{n}(\theta^\ast,\varepsilon_n; \theta_0)} e^{-\alpha\, r_{n}(\theta,\,\theta^\ast)} \Pi_n(d\theta) \leq e^{-\alpha\,( D + t) \,n\,\varepsilon_n^2}\Big] \leq \frac{1}{(D-1 +t)^2n\,\varepsilon_n^2}.
\end{align}

Now combining \eqref{Eq:post}, \eqref{Eqn:NumBound} and ~\eqref{Eqn:DenBound}, we obtain that with probability at least $1 - 2/\{(D-1+t)^2 n \varepsilon_n^2\}$,
\begin{align*}
\Pi_{n,\alpha} \Big(D^{(n)}_{\theta_0,\alpha}(\theta,\,\theta^\ast) \geq \frac{D+3t}{1-\alpha}\, n\,\varepsilon_n^2\ \Big| \, X^{(n)}\Big)
\leq  e^{-(D+2t)\, n \, \varepsilon_n^2}  e^{\alpha (D+t)\, n \, \varepsilon_n^2} \leq  e^{-t\, n \, \varepsilon_n^2}.
\end{align*}

\subsection{Proof of Corollary~\ref{Cor:contraction}}
An application of a union probability bound to Theorem~\ref{Thm:contraction} yields that 
\begin{align*}
\Pi_{n,\alpha} \Big(D^{(n)}_{\theta_0,\alpha}(\theta,\,\theta^\ast)\geq \frac{D+3j}{1-\alpha}\, n\,\varepsilon_n^2\ \Big| \, X^{(n)}\Big) \leq e^{-j\, n\, \varepsilon_n^2}, \qquad j=1,2,\ldots,
\end{align*}
holds with $\bbP_{\theta_0}^{(n)}$ probability at least $1 - 2 (n\,\varepsilon_n^2)^{-1} \sum_{j=1}^\infty 1/(D-1+j)^2 :\,=1 - C/\{n \varepsilon_n^2\}$, where the constant $C$ only depends on $D$. 
Since $\Pi_{n,\alpha} \big(D^{(n)}_{\theta_0,\alpha}(\theta,\,\theta^\ast) \geq \frac{D+3j}{1-\alpha}\, n\,\varepsilon_n^2\ \big| \,X^{(n)}\big)$ is non-increasing in $j$, we have, with $\bbP_{\theta_0}^{(n)}$ probability at least $1 - C/\{n \varepsilon_n^2\}$,
\begin{align*}
\Pi_{n,\alpha} \Big(D^{(n)}_{\theta_0,\alpha}(\theta,\,\theta^\ast)\geq \frac{D+3t}{1-\alpha}\, n\,\varepsilon_n^2\ \Big| \, X^{(n)}\Big)\leq e^{- \lfloor t \rfloor\, n\, \varepsilon_n^2} \leq  e^{-(t-1)\, n\, \varepsilon_n^2} \quad \mbox{for all $t\geq 1$}.
\end{align*}
Therefore, using this bound and applying Fubini's theorem, we obtain
\begin{align*}
\int \Big\{D^{(n)}_{\theta_0,\alpha}(\theta,\,\theta^\ast)\Big\}^k\,  \Pi_{n,\alpha} (d\theta\, \big| \, X^{(n)})  & =k \int_{0}^\infty \Pi_{n,\alpha} (D^{(n)}_{\theta_0,\alpha}(\theta,\,\theta^\ast) \geq u\, |\, X^{(n)})\, u^{k-1} du \\
& \leq \frac{k\, (D+6)^k}{(1-\alpha)^k}\, n^k\,\varepsilon_n^{2k} + k \int_{(D+6)\,n\,\varepsilon_n^2/(1-\alpha)}^\infty \Pi_{n,\alpha} (D^{(n)}_{\theta_0,\alpha}(\theta,\,\theta^\ast) \geq u\, |\, X^{(n)})\, u^{k-1} du \\
& \leq \frac{k\, (D+6)^k}{(1-\alpha)^k}\, n^2\,\varepsilon_n^{2k} + 3k\,n^k\, \varepsilon_n^{2k} e^{-n\varepsilon_n^2} \, \int_{2}^\infty e^{-(t-2)\, n\, \varepsilon_n^2} \frac{(D+3t)^{k-1}}{(1-\alpha)^k}\, dt\\
& \leq \frac{C}{(1-\alpha)^k}\, n^k\,\varepsilon_n^{2k} 
\end{align*}
for some constant $C$ depending on $k$.

\subsection{Proof of Theorem~\ref{Thm:fractional_convergence}}
Recall that the density of  $\Pi_{n,1}$ with respect to the prior $\Pi_n$ is
\begin{align*}
\frac{d\Pi_{n, \alpha}}{d\Pi_n}(\theta) = \frac{e^{-\alpha \,r_n(\theta,\theta_0)}}{\int e^{-\alpha \,r_n(\theta,\theta_0)} \Pi_n(d\theta)}.
\end{align*} 
Since $\bbE^{(n)}_{\theta_0}[\int e^{-r_n(\theta,\theta_0)} \, \Pi_n(d\theta)] = 1$, we have $\int e^{-r_n(\theta,\theta_0)} \Pi_n(d\theta) \in (0,\infty)$ almost surely. Let $E_n$ denote the set of all $X^{(n)}$ such that the integral is finite, then $\bbP^{(n)}_{\theta_0}(E_n) =1$.
We first prove that the denominator converges almost surely, that is,
\begin{align*}
\lim_{\alpha\to 1_{-}} \int e^{-\alpha \,r_n(\theta,\theta_0)}\, \Pi_n(d\theta) = \int e^{- r_n(\theta,\theta_0)} \, \Pi_n(d\theta), \quad a.s.
\end{align*}
In fact, for any $X^{(n)}\in E_n$, let 
\begin{align*}
A_n(X^{(n)}) = \{\theta\in\Theta:\, r_n(\theta,\theta_0) < 0\} \quad \mbox{and}\quad B_n(X^{(n)}) = \{\theta\in\Theta:\, r_n(\theta,\theta_0) \geq 0\},
\end{align*}
then we have the decomposition
\begin{align*}
 \int e^{-\alpha \,r_n(\theta,\theta_0)} \,\Pi_n(d\theta)  =  \int_{A_n(X^{(n)})} e^{-\alpha \,r_n(\theta,\theta_0)} \,\Pi_n(d\theta) +  \int_{B_n(X^{(n)})} e^{-\alpha\,r_n(\theta,\theta_0)} \,\Pi_n(d\theta).
\end{align*}
Since $e^{-\alpha \,s_n(f)}$ as a function of $\alpha$ is monotonically increasing in $A_n(X^{(n)})$ and bounded by $1$ in $B_n(X^{(n)})$, by applying the monotone convergence theorem for the first term and bounded convergence theorem for the second in the preceding display, we obtain
\begin{align} \label{Eqn:DEM}
\lim_{\alpha\to 1_{-}}  \int e^{-\alpha \,r_n(\theta,\theta_0)} \,\Pi_n(d\theta) =  \int e^{-r_n(\theta,\theta_0)} \,\Pi_n(d\theta) .
\end{align}
Similarly, for any  measurable set $B\in\mathcal{B}$, we have the decomposition 
\begin{align*}
 \int_B e^{-\alpha \,r_n(\theta,\theta_0)} \,\Pi_n(d\theta)  =  \int_{B\cap A_n(X^{(n)})} e^{-\alpha \,r_n(\theta,\theta_0)} \,\Pi_n(d\theta)+  \int_{B\cap B_n(X^{(n)})} e^{-\alpha \,r_n(\theta,\theta_0)} \,\Pi_n(d\theta).
\end{align*}
Therefore, by the monotone convergence theorem and the bounded convergence theorem, we have
\begin{align}\label{Eqn:NUM}
\lim_{\alpha\to 1_{-}} \int_B e^{-\alpha \,r_n(\theta,\theta_0)} \,\Pi_n(d\theta) = \int_B e^{-r_n(\theta,\theta_0)} \,\Pi_n(d\theta).
\end{align}
Combining \eqref{Eqn:DEM} and \eqref{Eqn:NUM}, we obtain that for any $X^{(n)}\in E_n$,
\begin{align*}
\Pi_{n,1}(B) = \lim_{\alpha\to 1_{-}} \Pi_{n,\alpha}(B).
\end{align*}
Since this is true for any measurable set $B\in\mathcal{B}$, we proved the theorem.

\subsection{Proof of Theorem~\ref{thm:renyi_main}}\label{Sec:Proof_thm:renyi_main}
We first state a key variational lemma that plays a critical role in the proof. 
\begin{lemma}\label{lem:var}
Let $\mu$ be a probability measure and $h$ a measurable function such that $e^h \in L_1(\mu)$. Then, 
\begin{align*}
\log \int e^h d\mu = \sup_{\rho \ll \mu} \bigg[ \int h d\rho - D(\rho \vert \vert \mu) \bigg].
\end{align*}
Further, the supremum on the right hand side is attained when 
$$
\frac{d\rho}{d\mu} = \frac{e^h}{\int e^h d \mu}.
$$
\end{lemma}
\begin{proof}
Fix $\rho \ll \mu$ and let $f_{\rho} = d\rho/d\mu$. Without loss of generality, we may assume $\mu \ll \rho$ so that $d\mu/d \rho = 1/f_{\rho}$, since otherwise we can always find a common dominating measure. Now, we have, by applying Jensen's inequality to the convex function $e^x$, that
\begin{align}\label{eq:key_divb}
\log \bigg[\int e^h d\mu \bigg] = \log \bigg[\int e^h \frac{d\mu}{d\rho} d\rho \bigg] = \log \bigg[ \int e^{h - \log f_{\rho}} d \rho \bigg] \ge \int h d\rho - \int \frac{d\rho}{d \mu} d \rho= \int h d\rho - D(\rho \vert \vert \mu).
\end{align}
Further, we have equality when $\log f_{\rho} - h$ is constant, or equivalently, when $d \rho/d \mu \propto e^h$.
\end{proof}

Return to the proof of the theorem. Recall the definition of the $\alpha$-Renyi divergence and $\alpha$-affinity that 
\begin{align*}
\bbE^{(n)}_{\theta_0} e^{-\alpha \, r_{n}(\theta,\theta_0)} = A^{(n)}_{\alpha}(\theta,\theta_0) = e^{- (1 - \alpha) D^{(n)}_{\alpha}(\theta, \theta_0) }.
\end{align*}
Thus, for any $\varepsilon \in (0, 1)$, we have
\begin{align*}
\bbE^{(n)}_{\theta_0} \exp \bigg[-\alpha \, r_{n}(\theta,\theta_0) + (1 - \alpha) D^{(n)}_{\alpha}(\theta, \theta_0) - \log(1/\varepsilon) \bigg] \le \varepsilon. 
\end{align*}
Integrating both side of this inequality with respect to $\Pi_n$ and interchanging the integrals using Fubini's theorem, we obtain
\begin{align*}
\bbE^{(n)}_{\theta_0} \int \exp\bigg[-\alpha \, r_{n}(\theta,\theta_0) + (1 - \alpha) D^{(n)}_{\alpha}(\theta, \theta_0) - \log(1/\varepsilon)\bigg] \Pi_n(d\theta) \le \varepsilon. 
\end{align*}
Now, Lemma \ref{lem:var} implies
\begin{align*}
\bbE^{(n)}_{\theta_0} \exp \sup_{\rho \ll \Pi_n} \bigg[\int \bigg\{ -\alpha \, r_{n}(\theta,\theta_0) + (1 - \alpha) D^{(n)}_{\alpha}(\theta, \theta_0)  - \log(1/\varepsilon)\bigg\} \rho(d\theta) - D(\rho \,\vert \vert \,\Pi_n)  \bigg] \le \varepsilon.
\end{align*}
If we choice $\rho(\cdot) = \Pi_{n, \alpha}(\cdot\,|\,X^{(n)})$ as the fractional posterior distribution in the preceding display, then
\begin{align*}
\bbE^{(n)}_{\theta_0} \exp \bigg[\int \bigg\{ -\alpha \, r_{n}(\theta,\theta_0) + (1 - \alpha) D^{(n)}_{\alpha}(\theta, \theta_0)  - \log(1/\varepsilon)\bigg\} \Pi_{n, \alpha}(d\theta \,|\,X^{(n)}) - D(\Pi_{n, \alpha} \,\vert \vert\, \Pi_n)  \bigg] \le \varepsilon. 
\end{align*}
By applying Markov's inequality, we further obtain that with $\bbP^{(n)}_{\theta_0}$ probability at least $(1 - \varepsilon)$, 
\begin{align*}%\label{eq:bd1}
(1 - \alpha) \int D^{(n)}_{\alpha}(\theta, \theta_0) \, \Pi_{n, \alpha}(d\theta \,|\,X^{(n)})
& \le \alpha  \int r_{n}(\theta,\theta_0) \, \Pi_{n, \alpha}(d\theta \,|\,X^{(n)})+ D(\Pi_{n, \alpha}(\cdot\,|\,X^{(n)})\, \vert \vert\, \Pi_n) + \log(1/\varepsilon).
\end{align*}
Noting the relation between $r_n(\theta,\theta_0)$ and $\Pi_{n, \alpha}(\cdot \,|\,X^{(n)})$ by \eqref{eq:frac_post} with $\theta^\dagger=\theta_0$, we can apply Lemma \ref{lem:var} with $h = - \alpha\, r_{n}(\theta,\theta_0)$ to obtain that
\begin{align*}
  \int  \alpha\, r_{n}(\theta,\theta_0) \, \Pi_{n, \alpha}(d\theta \,|\,X^{(n)})+ D(\Pi_{n, \alpha}(\cdot\,|\,X^{(n)})\, \vert \vert\, \Pi_n) 
& = \inf_{\rho \ll \Pi_n} \bigg[\int  \alpha\, r_{n}(\theta,\theta_0) \, \rho(d\theta) + D(\rho\, \vert \vert \,\Pi_n) \bigg].
\end{align*}
The second part is a direct consequence by combining the results in the two preceding displays.

\subsection{Proof of Corollary~\ref{cor:renyi}} \label{Sec:Proof_cor:renyi}
The second part is a direct consequence of the first part, which we are going to prove.
Pick $\varepsilon = e^{-\alpha n \varepsilon^2}$ in Theorem \ref{thm:renyi_main} and let $\m A \in \sigma(X_1, \ldots, X_n)$ denote the set on which \eqref{eq:renyi_bd} holds. Picking $\rho$ as $\rho^\ast=\Pi_n(\cdot \cap B_n(\theta_0,\varepsilon;\theta_0))/\Pi_n(B_n(\theta_0,\varepsilon;\theta_0))$, the restriction of the prior $\Pi_n$ to $B_n(\theta_0,\varepsilon;\theta_0)$. Let $\m A' \in \sigma(X_1, \ldots, X_n)$ denote the event in which
$$\int r_n(\theta,\theta_0)  \rho^\ast(d\theta) \le D  n \, \varepsilon^2$$
holds.
 On the set $\m A \cap \m A'$, \eqref{eq:cor_renyi} holds. It thus remains to bound $\bbP^{(n)}_{\theta_0}(\m A \cap \m A') \ge \bbP^{(n)}_{\theta_0}(\m A) + \bbP^{(n)}_{\theta_0}(\m A') - 1$ from below. 
We now show that $\bbP^{(n)}_{\theta_0}(\m A') \ge 1 - 1/\{(D-1)^2 n \varepsilon^2\}$, which completes the proof, since $e^{- \alpha n \varepsilon^2} \le 1/\{(D-1)^2 n \varepsilon^2\}$. By applying Fubini's theorem and invoking the definition of $B_n(\theta_0,\varepsilon;\theta_0)$, we have
\begin{align*}
\bbE_{\theta_0}^{(n)}\Big[\int r_n(\theta,\theta_0)  \rho^\ast(d\theta)\Big] = \int  \bbE_{\theta_0}^{(n)}\big[ r_n(\theta,\theta_0) \big]\, \rho^\ast(d\theta)\leq n\,\varepsilon^2.
\end{align*} 
Thus, by applying Cauchy-Schwarz inequality, Chebyshev's inequality and Fubini's theorem, we have
\begin{align*}
& \bbP^{(n)}_{\theta_0} \{ (\m A')^c\}
= \bbP^{(n)}_{\theta_0}\bigg( \int r_n(\theta,\theta_0)  \rho^\ast(d\theta) > Dn \varepsilon^2 \bigg) \\
& \le  \bbP^{(n)}_{\theta_0}\bigg\{  \int r_n(\theta,\theta_0)  \rho^\ast(d\theta) - \bbE_{\theta_0}^{(n)}\Big[  \int r_n(\theta,\theta_0)  \rho^\ast(d\theta)\Big]> (D-1)n \varepsilon^2\bigg\}  \\
& \le \frac{\bbE^{(n)}_{\theta_0}  \big[ \int r_n(\theta,\theta_0)  \rho^\ast(d\theta)\big]^2 }{(D-1)^2 n^2 \varepsilon^4} 
\le \frac{  \int \bbE^{(n)}_{\theta_0}\big[r_n(\theta,\theta_0)]^2\, \rho^\ast(d\theta) }{(D-1)^2 n \varepsilon^4}
\le \frac{1}{(D-1)^2n \varepsilon^2},
\end{align*}
where in the last step we used the definition of  $B_n(\theta_0,\varepsilon;\theta_0)$.

\subsection{Proof of Theorem~\ref{thm:renyi_mainb}}
The proof follows the same lines as the proof of Theorem~\ref{thm:renyi_main} in Section~\ref{Sec:Proof_thm:renyi_main}. The only difference is that we use~\eqref{eq:key_div} instead of \eqref{eq:key_divb} when applying the variational lemma to obtain the PAC-Bayes inequality. Due to space constraints, we omit the proof.

\subsection{Proof of Corollary~\ref{cor:renyi_b}}
The proof follows the same lines as the proof of Corollary~\ref{cor:renyi} in Section~\ref{Sec:Proof_cor:renyi}. The only difference is that we use $B_n(\theta^\ast,\varepsilon;\theta_0)$ instead of $B_n(\theta_0,\varepsilon;\theta_0)$ in the definition of $\rho$ and in the application of Chebyshev's inequality. Due to space constraints, we omit the proof.

\subsection{Proof of Theorem \ref{thm:conv:rate}}
In the well-specified case, we have $\mu^\ast=\mu_0$. We choose $\alpha =1/2$ so that  
$D^{(n)}_{\theta_0, 1/2}(\theta, \theta_0) = 
\norm{\mu- \mu_0}_{2,n}^2/(4\sigma^2)$ and 
\begin{align}
 \int p_{\theta_0}^{(n)} \log (p_{\theta_0}^{(n)}/p_{\theta}^{(n)}) d\mu^{(n)} \leq  \norm{\mu_0 - \mu}_{2,n}^2/(2\sigma^2), \,
\int p_{\theta_0}^{(n)} \log^2 (p_{\theta_0}^{(n)}/p_{\theta}^{(n)})  d\mu^{(n)}\leq 
 \norm{\mu_0 - \mu}_{2,n}^2/\sigma^2. 
\end{align}
$B_n(\theta_0,\varepsilon;\theta_0)$ can be re-written as 
$
\{ \norm{\mu - \mu_0}_{2,n}^2 \leq 2\sigma^2 \varepsilon_n^2 \}. 
$
It suffices to obtain a lower bound on $\Pi(\norm{\mu - \mu_0}_{\infty} < C\varepsilon)$ for some constant $C$, since it implies a same lower bound on  $\Pi(\norm{\mu - \mu_0}_{2,n} < C\varepsilon)$.  Fix $\varepsilon > 0$.  Using Lemma 4.1 of \cite{balazs2015near},  we obtain a sequence of numbers $\{p^k_j \in \mathbb{R}^d, j=1, \ldots, k\}$ and $\{q_k^j \in \mathbb{R}\, j=1, \ldots, k\}$ with $||p^k_j|| \leq A_1$ and  $|q^k_j| \leq B_1$ where $A_1$ and $B_1$ are global constants depending on $d$ such that for $k > \lceil (\varepsilon/2)^{-d/2}\rceil$, with $\tilde{\mu}^k(x) = \max_{j \in \{1, \ldots, k\}} \big[(p^k_j)^{\T} x  +  q^k_j \big]$ we have $||\mu_0- \tilde{\mu}^k||_{\infty} \leq \varepsilon/2$. Observe that for any $k > \lceil (\varepsilon/2)^{-d/2}\rceil$, there exist constants $0 < \delta_1, \delta_2 < 1$, dependent only on $d$ such that for any sequence $\{a_j^k \in \mathbb{R}^d, j=1, \ldots, k\}$ and $\{b_j^k \in \mathbb{R}, j=1, \ldots, k  \}$ satisfying 
$\max_{j \in \{1,\ldots, k\}} ||a_j^k - p_j^k|| < \delta_1 \varepsilon$ and $\max_{j \in \{1,\ldots, k\}} ||b_j^k - q_j^k|| < \delta_2 \varepsilon$,   we have $
||\mu- \tilde{\mu}^k||_{\infty} \leq \varepsilon/2$ for $\mu =   \max_{j \in \{1, \ldots, k\}} \big[(a^k_j)^{\T} x  +  b^k_j \big]$. 
This is possible since $||p^k_j|| \leq A_1$ and  $|q^k_j| \leq B_1$.  From Anderson's inequality \cite{anderson1973normal} and standard reslts on centered small ball probability of multivariate Gaussian random variables, we obtain for sufficiently small $\varepsilon$, 
\begin{eqnarray*}
 \Pi\bigg(\max_{j \in \{1,\ldots, k\}} ||a_j^k - p_j^k|| < \delta_1 \varepsilon\bigg) \geq e^{- C_1 d k \log (1/\varepsilon)}, \quad 
 \Pi\bigg(\max_{j \in \{1,\ldots, k\}} ||a_j^k - p_j^k|| < \delta_2 \varepsilon\bigg) \geq  e^{- C_2 k \log (1/\varepsilon)},
\end{eqnarray*}
To lower bound $\Pi(\norm{\mu - \mu_0}_{\infty} < C\varepsilon)$, we consider the sum over  $k \in  [\lceil (\varepsilon/2)^{-d/2}\rceil,  2\lceil (\varepsilon/2)^{-d/2}\rceil]$. Using  $\pi(k) \geq \exp\{-Ck \log k\}$, we have $\Pi(\norm{\mu - \mu_0}_{\infty} < C\varepsilon) \geq e^{- C_3 \varepsilon_n^{-d/2} \log (1/\varepsilon)}$ for sufficiently small $\epsilon$.  Hence $\Pi(\norm{\mu - \mu_0}_{\infty} < C\varepsilon_n) \geq e^{-C_4 n \varepsilon_n^2}$ is satisfied for $\varepsilon_n =  n^{-2/(4+d)} (\log n)^t$ for $t = 2/(4+d)$.

\subsection{Proof of Theorem \ref{thm:conv:rate_mis}}
We apply Corollary~\ref{cor:renyi_b} to obtain that with probability tending to one,
\begin{align}
&\int  \frac{\alpha}{2(1-\alpha)} \bigg[ \|\mu-\mu_0\|_{2,n}^2 -  \|\mu^\ast-\mu_0\|_{2,n}^2 - \alpha\,  \|\mu-\mu^\ast\|_{2,n}^2\bigg] d\Pi_{n, \alpha}(\mu)\notag \\
&\leq  \frac{D\, \alpha}{1-\alpha} \, \varepsilon^2  + \Big\{ - \frac{1}{n(1-\alpha)} \log \Pi(B_n(\mu^\ast,\varepsilon;\mu_0)) \Big\}.\label{Eqn:convex_mis}
\end{align}
By the convexity of $\mbox{Co}_L[0,1]^d$ and the definition of $\mu^\ast$ as the projection, we have that for any $\mu\in\mbox{Co}_L[0,1]^d$,
\begin{align*}
 &\|\mu-\mu_0\|_{2,n}^2 -  \|\mu^\ast-\mu_0\|_{2,n}^2 - \alpha\,  \|\mu-\mu^\ast\|_{2,n}^2 = (1-\alpha) \,  \|\mu-\mu^\ast\|_{2,n}^2 + 2\langle \mu -\mu^\ast, \mu^\ast-\mu_0\rangle_n\\
 & \geq  (1-\alpha) \,  \|\mu-\mu^\ast\|_{2,n}^2.
\end{align*}
Therefore, inequality~\eqref{Eqn:convex_mis} in particular implies
\begin{align*}
\int  \frac{\alpha}{2} \, \|\mu-\mu^\ast\|_{2,n}^2 d\Pi_{n, \alpha}(\mu)\leq  \frac{D\, \alpha}{1-\alpha} \, \varepsilon^2  + \Big\{ - \frac{1}{n(1-\alpha)} \log \Pi(B_n(\mu^\ast,\varepsilon;\mu_0)) \Big\}.
\end{align*}
Plugging this back into~\eqref{Eqn:convex_mis}, and noting that $ \|\mu^\ast-\mu_0\|_{2,n}^2$ is independent of $\mu$, we obtain
\begin{align*}
\int \|\mu-\mu_0\|_{2,n}^2 d\Pi_{n, \alpha}(\mu)\notag \leq\inf_{\mu\in \mbox{Co}_L[0, 1]^d}  \norm{\mu - \mu_0}_{2, n}^2 +  \frac{D}{(1-\alpha)}\varepsilon^2 + \Big\{  - \frac{C\,}{n\,\alpha\, (1-\alpha)} \log \Pi(B_n(\mu^\ast,\varepsilon;\mu_0)) \Big\}
\end{align*}
for some constant $C>0$.

Using the expression~\eqref{Eqn:formula_convex}, we have
\begin{align*}
B_n(\mu^\ast, \varepsilon; \mu_0) \supset \{\|\mu-\mu^\ast\|_{2,n} \leq C_1\, \varepsilon^2\},
\end{align*}
for some constant $C_1$ independent of $n$. Now use the prior concentration results derived in the proof of Theorem \ref{thm:conv:rate}, we obtain
\begin{align*}
\Pi\big(B_n(\mu^\ast, \varepsilon; \mu_0)\big) \geq  e^{-C_2\, n\, \varepsilon_n^{2} \log(1/\varepsilon)},
\end{align*}
for some constant $C_2>0$.
Putting pieces together, we obtain by choosing $\varepsilon = \varepsilon_{n}$ that
\begin{align*}
\int \|\mu-\mu_0\|_{2,n}^2 d\Pi_{n, \alpha}(\mu)\notag \leq\inf_{\mu\in \mbox{Co}_L[0, 1]^d}  \norm{\mu - \mu_0}_{2, n}^2 +  \frac{C_3 }{\alpha\, (1-\alpha)}\varepsilon_n^2,
\end{align*}
implying the desired Bayesian oracle inequality.

\subsection{Proof of Theorem \ref{thm:npreg:rate}}
Similar to the proof of Theorem \ref{thm:conv:rate}, it suffices to obtain a lower bound on $\Pi(\norm{\mu - \mu_0}_{\infty} < C_1\varepsilon)$.  From Section 5.1 (pp. 2671) of \cite{van2009adaptive}, it can be shown that for $\varepsilon$ sufficiently small
\begin{eqnarray*}
\Pi(\norm{\mu - \mu_0}_{\infty} < C_1\varepsilon)  \geq C_2 e^{-C_3(1/\varepsilon)^{d/\alpha} \{\log (1/\varepsilon) \}^{(1+d)\vee q}}.  
\end{eqnarray*}
Therefore, $\Pi(\norm{\mu - \mu_0}_{\infty} < C_1\varepsilon_n) \geq e^{-n \varepsilon_n^2}$ is satisfied with 
$\varepsilon_n \geq n^{-\alpha/(2\alpha +d)} (\log n)^t$ where $t = \{(1+d) \vee q\}/ (2 + d/\alpha)$.  
From Corollary \ref{cor:renyi} with $\alpha =1/2$ and $\varepsilon_n = n^{-\alpha/(2\alpha +d)} (\log n)^t$, we have,
\begin{align}\label{eq:cor_renyiN}
 \int \norm{\mu - \mu_0}_{2,n}^2 \Pi_{n, \alpha}(d\mu) \le  C_4 \varepsilon_n^2
\end{align}
with $\bbP^{(n)}_{\theta_0}$ probability at least $1 - 2/\{(D-1)^2 n \varepsilon_n^2\}$ for some  $D > 1$.

\subsection{Proof of Theorem \ref{th:rate_cond_dens}}
The  proof of Theorem \ref{th:rate_cond_dens}  follows in a straightforward manner from 
\cite{norets2016} with the following modifications. 
To apply Corollary \ref{cor:renyi},  we need to prove that for the $\varepsilon$ stated in the statement of Theorem  \ref{th:rate_cond_dens}, we have $-\log \Pi_n(B_{n}(\theta_0,\varepsilon_n; \theta_0))\leq  n \varepsilon_n^2$.  
For  $\sigma_n = \varepsilon_n / \log (1/ \varepsilon_n) ]^{1/\beta}$,
$\varepsilon$ defined in \eqref{eq:asnE0Dff0_Lf0},
a sufficiently small $\delta>0$,
$b$ and $\tau$ defined in \eqref{eq:asnf0_exp_tails},
$a_0 = \{ (8\beta + 4\varepsilon +16)/(b \delta)\}^{1/\tau}$,
$a_{\sigma_n} = a_0 \{\log (1/\sigma_n) \}^{1/\tau}$, 
and $b_1 > \max \{1, 1/ 2\beta \}$ satisfying $\varepsilon_n^{b_1}  \{  \log (1/ \varepsilon_n) \}^{5/4} \leq \varepsilon_n$,
the proof of Theorem 4 in \cite{ShenTokdarGhosal2013} implies the following three claims.
First,
there exists a partition of $\{y \in \mathcal{Y}: ,y, \leq a_{\sigma_n}\}$,
$\{U_j, j=1,\ldots,K\}$ such that for $j=1,\ldots,N$,
$U_j$ is a ball with diameter $\sigma_n \varepsilon_n^{2 b_1}$
and center $y_j$;
for $j=N+1,\ldots,K$,  
$U_j$ is a set with a diameter bounded above by $\sigma_n$;
$1 \leq N < K \leq C_2 \sigma_n^{-d} \{\log (1/ \varepsilon_n) \}^{d +d/\tau}$, 
where $C_2>0$ does not depend on $n$.
Second, there exist 
$\theta^\star = \{\mu_j^\star, \alpha_j^\star, j = 1,2,\ldots; \sigma_n\}$ with
$\alpha_j^\star=0$ for $j > N$, $\mu_j^\star=z_j$ for $j=1,\ldots,N$, and
$\mu_j^\star \in U_j$ for $j=N+1,\ldots,K$
 such that 
for $m=K$ and a positive constant $C_3$,
\begin{equation}
\label{eq:f0upsbeta}
h(f_0 , p(\cdot |\theta^\star, m) ) \leq C_3 \sigma_n^\beta.
\end{equation}
Third, there exists constant $B_0>0$ such that
\begin{equation}
\label{eq:P0_zgeqa}
P_0(\norm{Y} > a_{\sigma_n}) \leq B_0 \sigma_n^{4\beta + 2\varepsilon +8}.
\end{equation}
For $\theta \in S_{\theta^\star}$, where,
\begin{align*}
	S_{\theta^\star}=&\big \{
	(\mu_j, \alpha_j, \, j=1,2,\ldots; \sigma): 
	 \; \mu_j  \in U_j, j=1, \ldots, K, \\
	& \sum_{j=1}^K \abs{\alpha_j - \alpha_j^\star} \leq 2\varepsilon_n^{2db_1}, \, \min_{j=1, \ldots, K} \alpha_j \geq \varepsilon_n^{4db_1}/2, 
	%\\	& 
	\sigma^2 \in [ \sigma_n^{2}/(1+ \sigma_n^{2\beta}), \sigma_n^{2} ]
\big \},
\end{align*}
 we have 
\begin{align*}
& d_h^2(p(\cdot | \theta^\star, m),  p(\cdot | \theta, m))  \leq 
\norm{\sum_{j=1}^K \alpha_j^\star \phi_{\mu_j^\star,\sigma_n} - \sum_{j=1}^K \alpha_j \phi_{\mu_j,\sigma}}_1\\
&\leq  \sum_{j=1}^K \abs{\alpha_j^\star-\alpha_j} 
+ \sum_{j=1}^N \alpha_j^\star 
\left[\norm{\phi_{\mu_j^\star,\sigma_n} - \phi_{\mu_j,\sigma_n}}_1
+
\norm{\phi_{\mu_j,\sigma_n} - \phi_{\mu_j,\sigma}}_1 \right].
%\\
%&\leq&   2\varepsilon_n^{2d b_1} + \sum_{j=1}^N p_j \frac{\norm{\mu_j - z_j}}{\sigma}  \leq   2\varepsilon_n^{2b_1}. 
\end{align*}
For $j=1,\ldots,N$,
$
\norm{\phi_{\mu_j^\star,\sigma_n} - \phi_{\mu_j,\sigma_n}}_1 \leq ,\mu_j^\star - \mu_j,/\sigma_n \leq 
\varepsilon_n^{2b_1}
$.  Also, 
\begin{equation}
\label{eq:bd4sigmas}
\norm{\phi_{\mu_j, \sigma_n} - \phi_{\mu_j, \sigma}}_1  
\leq \sqrt{d/2} 
\abs{\frac{\sigma_n^2}{\sigma^2} - 1 - \log \frac{\sigma_n^2}{\sigma^2}}^{1/2} \leq  C_4 \sqrt{d/2} \abs{\frac{\sigma_n^2}{\sigma^2} - 1}
\lesssim  \sigma_n^{2\beta}, 
\end{equation}
where the penultimate inequality follows from the fact that 
$\abs{\log x - x +1} \leq C_4\abs{x-1}^2$ for $x$ in a neighborhood of 1 and some $C_4 > 0$.   Hence,   
 $d_h (p(\cdot | \theta, m), p(\cdot | \theta^\star, m)) \leq C_5 \sigma_n^{\beta}$for some $C_5>0$, all $\theta \in S_{\theta^\star}$, and $m=K$.

Next, for $\theta \in S_{\theta^\star}$, let us consider a lower bound on the ratio 
$p(y | \theta, m)/f_0(y)$.  
Note that $\sup_{y} f_0(y)<\infty$ and
$p(y | \theta, m) \geq \sigma^{d} p(y| \theta, m)$.
For $y \in \mathcal{Y}$ with $\norm{z} \leq a_{\sigma_n}$, there exists $J \leq K$ for which 
$,y-\mu_J , \leq \sigma_n$. Thus, 
for all sufficiently large $n$ such that $\sigma_n^2/\sigma^2 \leq 2$,
$p(y | \theta, m) \geq \min_j \alpha_j \cdot \phi_{\mu_J,\sigma} (y) \geq
[\varepsilon_n^{4db_1}/2] \cdot \sigma_n^{-d} e^{-1} / (2\pi)^{d/2}$ and 
\begin{equation}
\label{eq:lambda_def}
\frac{p(y |\theta, m)}{f_0(y)} \geq
C_6 \varepsilon_n^{4db_1} \sigma_n^{-d}, \mbox{ for some } C_6>0.
\end{equation}
For $y \in \mathcal{Y}$ with $\norm{y} > a_{\sigma_n}$, 
$\norm{ z - \mu_j}^2 \leq 2 (\norm{z}^2 + \norm{\mu}^2) \leq 4\norm{y}^2$
for all $j=1,\ldots,K$. Thus, for all sufficiently large $n$,
$
p(y | \theta, m) \geq \sigma_n^{-d} \exp (-4\norm{y}^2/\sigma_n^2) / (2\pi)^{d/2}
$ and 
\begin{equation*}
\frac{p(y | \theta, m)}{f_0(y)} \geq
C_7 \sigma_n^{-d_y} \exp (-4\norm{z}^2/\sigma_n^2), \mbox{ for some } C_7>0.
\end{equation*}

Denote the lower bound in \eqref{eq:lambda_def} by $\lambda_n$ and consider all sufficiently large $n$ such that 
$\lambda_n < e^{-1}$.
For any $\theta \in S_{\theta^\star}$, 
\begin{align*}
& \int \bigg( \log \frac{f_0(y)}{p(y|\theta, m)}\bigg)^2 
1\left \{ \frac{p(y|\theta, m)}{f_0(y)} < \lambda_n \right\}
f_0(y) g_0(x) dydx \\
& =
\int \bigg( \log \frac{f_0(y)}{p(y|\theta, m)}\bigg)^2 1\left \{ \frac{p(y|\theta, m)}{f_0(y)} < \lambda_n, , (y,x), > a_{\sigma_n} \right\} f_0(y) dy
\\
& \leq
\frac{4}{\sigma_n^4} \int_{\norm{y} > a_{\sigma_n}} 
\norm{y}^4 f_0(y) dy \leq \frac{4}{\sigma_n^4} E_0(\norm{Y}^8)^{1/2} (P_0(\norm{Y} > a_{\sigma_n}))^{1/2}  \leq C_8 \sigma_n^{2\beta + \varepsilon},
 \end{align*}
 for some constant $C_8$.  The last inequality follows from 
\eqref{eq:P0_zgeqa} and tail condition in \eqref{eq:asnf0_exp_tails}.
Also note that 
 \begin{eqnarray*}
 \log \frac{f_0(y)}{p(y | \theta, m)} 1\left\{\frac{p(y| \theta, m)}{f_0(y)} < \lambda_n \right\} \leq \left\{ \log 
\frac{f_0(y)}{p(y| \theta, m)}\right\}^2 1\left\{\frac{p(y|\theta, m)}{f_0(y)} < \lambda_n \right\}
 \end{eqnarray*}
 and, thus, 
 \begin{eqnarray*}
 \int \log \frac{f_0(y)}{p(y|x, \theta, m)} 1 \left \{ \frac{p(y|\theta, m)}{f_0(y)} < \lambda_n \right \} f_0(y) dy  \leq C_8\sigma_n^{2\beta + \varepsilon}.
 \end{eqnarray*}
From Lemma 4 of \cite{ShenTokdarGhosal2013}, both
$E_0 (\log (f_0(Y)/p(Y|\theta, m)))$ and $E_0 ([\log (f_0(Y)/p(Y|\theta, m))]^2)$ are bounded by 
$C_9 \log (1/\lambda_n)^2\sigma_n^{2\beta} \leq A\varepsilon_n^2$ for some constant $A$.  
Finally, we calculate a lower bound on the prior probability of $m = K$ and $\{\theta \in S_{\theta^\star}\}$.
By \eqref{eq:asnPrior_m}, for some $C_{10}>0$,
\begin{equation}
\Pi(m=K) \propto \exp[-a_{10} K (\log K)^{\tau_1}] \geq \exp [-C_{10}\varepsilon_n^{-d/\beta} \{\log (1/ \varepsilon_n)\}^{d + d/\beta+ d/\tau + \tau_1 }] \label{eq:KL1}. 
\end{equation}
From Lemma 10 of \cite{GhosalVandervaart:07}, for some constants $C_{11}, C_{12}>0$ and all sufficiently large $n$,
\begin{align}
&\Pi\left( \sum_{j=1}^K \abs{\alpha_j - \alpha_j^\star} \geq 2 \varepsilon_n ^{2db_1}, \min_{j=1, \ldots, K} \alpha_j \geq \varepsilon_n^{4db_1}/2 \bigg| m=K\right) \geq \exp[ -C_{11} K \log (1/\varepsilon_n)]  
\notag \\
&\geq 
\exp [-C_{12} \varepsilon_n^{-d/\beta} \{\log (1/ \varepsilon_n)\}^{d/\beta + d/\tau +d +1}].
\end{align}
For $\pi_{\mu}$ denoting the prior density of $\mu_j^y$ and some $C_{13}, C_{14}>0$, \eqref{eq:asnPrior_mu_lb} implies
\begin{align}
&\Pi(\mu_j \in U_j, j=1, \ldots, N ) \geq \{C_{13}\pi_{\mu}(a_{\sigma}) \mbox{diam}(U_1)^d \}^N
\notag \\ 
&\geq 
\exp \left [ -C_{14} \varepsilon_n^{-d/\beta}
\left \{\log (1/ \varepsilon_n)\right\}^{d + d/\beta + d/\tau +\max\{1, \tau_2/\tau\}}\right].
\label{eq:Pr_muU}
\end{align}
Assumption \eqref{eq:asnPrior_sigma3} on the prior for $\sigma$, implies
%For the assumed inverse gamma prior for $\sigma$, the mean value theorem implies
\begin{equation}
\Pi(\sigma^{-2} \in \{ \sigma_n^{-2}, \sigma_n^{-2}(1+ \sigma_n^{2\beta}) \}) \geq a_{8} \sigma_n^{-2 a_7} 
 \sigma_n^{2\beta a_8} \exp\{ - a_9\sigma_n^{-1} \} \geq
\exp\{ - C_{15}\sigma_n^{-1} \}\label{eq:KL4}. 
\end{equation} 
It follows from \eqref{eq:KL1} - \eqref{eq:KL4},  
that for all sufficiently large $n$,  $s = 1 + 1/\beta + 1/\tau$, and some $C_{16}>0$ 
\begin{eqnarray*}
\Pi(B_n(\theta_0,\varepsilon_n;\theta_0))  \geq \Pi(m=N, \theta_p \in S_{\theta_p} )  \geq 
\exp [-C_{16}\varepsilon_n^{-d/\beta} \{\log (1/ \varepsilon_n)\}^{ds + \max\{\tau_1,1,\tau_2/\tau\} }] .
\end{eqnarray*}
The last expression of the above display is bounded below by $\exp\{-C n \varepsilon_n^2 \}$ for any $C>0$,
$\varepsilon_n =  n^{-\beta/(2\beta + d)} (\log n)^t$, any $t > (ds + \max\{\tau_1,1,\tau_2/\tau\}) / (2 + d/\beta)$, and all sufficiently large $n$.  Since the inequality in the definition of $t$ is strict, the claim of the theorem follows immediately.

\bibliography{NSF_refs}

\end{document}